\documentclass[10pt,reqno]{amsart}
\usepackage{amsmath}
\usepackage{amssymb}
\usepackage{amsthm}
\usepackage{eepic,epic}
\usepackage{epsfig}
\usepackage{graphicx}

\newlength{\defbaselineskip}
\newcommand{\setlinespacing}[1]%
           {\setlength{\baselineskip}{#1 \defbaselineskip}}

\numberwithin{equation}{section}

\newtheorem{thm}{Theorem}[section]

\newtheorem{lem}[thm]{Lemma}
\newtheorem{prop}[thm]{Proposition}
\theoremstyle{definition}

\theoremstyle{remark}
\newtheorem{rem}[thm]{Remark}
\numberwithin{equation}{section}

\begin{document}

\title[Strichartz and smoothing estimates]
{Strichartz and smoothing estimates in weighted $L^2$ spaces and their applications}

\author{Youngwoo Koh and Ihyeok Seo}


\thanks{Y. Koh was supported by NRF-2019R1F1A1054310.
I. Seo was supported by NRF-2019R1F1A1061316.}

\address{Department of Mathematics Education, Kongju National University, Kongju 32588, Republic of Korea}
\email{ywkoh@kongju.ac.kr}

\address{Department of Mathematics, Sungkyunkwan University, Suwon 16419, Republic of Korea}
\email{ihseo@skku.edu}

\maketitle


\begin{abstract}
The primary objective in this paper is to give an answer to an open question posed by
J. A. Barcel\'o, J. M. Bennett, A. Carbery, A. Ruiz and M. C. Vilela (\cite{BBCRV2})
concerning the problem of determining the optimal range on $s\geq0$ and $p\geq1$ for which the following Strichartz estimate
with time-dependent weights $w$ in Morrey-Campanato type classes $\mathfrak{L}^{2s+2,p}_2$ holds:
\begin{equation}\label{absset}
\|e^{it\Delta}f\|_{L_{x,t}^2(w(x,t))}\leq C\|w\|_{\mathfrak{L}^{2s+2,p}_2}^{1/2}\|f\|_{\dot{H}^s}.
\end{equation}
Beyond the case $s\geq0$, we further ask how much regularity we can expect on the setting \eqref{absset}.
But interestingly, it turns out that \eqref{absset} is false whenever $s<0$,
which shows that the smoothing effect cannot occur in this time-dependent setting
and the dispersion in the Schr\"odinger flow $e^{it\Delta}$ is not strong enough to have the effect.
This naturally leads us to consider the possibility of having the effect at best
in higher-order versions of \eqref{absset} with $e^{-it(-\Delta)^{\gamma/2}}$ ($\gamma>2$) whose dispersion is more strong.
We do obtain a smoothing effect exactly for these higher-order versions.
In fact, we will obtain the estimates where $\gamma\geq1$ in a unified manner and also their corresponding inhomogeneous estimates
to give applications to the global well-posedness for Schr\"odinger and wave equations with time-dependent perturbations.
This is our secondary objective in this paper.
\end{abstract}

\section{Introduction}
Over the last three decades, Strichartz and smoothing estimates proved to be very efficient tools
for dealing with dispersive equations.
The former initiated in \cite{Str2}
is useful for solving equations in which no derivatives are present in the potential or the nonlinearity,
while the latter is particularly important in other cases
since it makes it possible to recover the loss of derivatives in the equations.
This was an observation of Kato \cite{Ka}
when he succeeded in solving the Korteweg-de Vries (KdV) equation, $\partial_t+\partial_x^3u+u\partial_xu=0$,
by establishing a local smoothing estimate which shows that the solution is, locally, one derivative smoother than $L^2$ initial data.
For these reasons, those estimates have been intensively investigated for various equations of great importance
in mathematical physics.

In this paper we study Strichartz and smoothing estimates for a class of dispersive equations
\begin{equation}\label{disp}
\begin{cases}
i\partial_tu-(-\Delta)^{\gamma/2}u=0,\\
u(x,0)=f(x),
\end{cases}
\end{equation}
where $(-\Delta)^{\frac {\gamma}{2}}$ is defined for $\gamma\geq1$
by means of the Fourier transform as
$$\widehat{[(-\Delta)^{\frac {\gamma}{2}} f]}(\xi)=|\xi|^\gamma \widehat{f}(\xi),$$
and applying the Fourier transform to \eqref{disp}, the solution is then given by
\begin{equation*}
e^{-it(-\Delta)^{\gamma/2}}f(x) :=(2\pi)^{-n}\int_{\mathbb{R}^n}e^{ix\cdot\xi}e^{-it|\xi|^\gamma} \widehat{f}(\xi)d\xi.
\end{equation*}
The motivation behind \eqref{disp} is that its solution operator $e^{-it(-\Delta)^{\gamma/2}}$ can be used to describe solutions
for various equations like Schr\"odinger and wave equations, as will be seen below (see Subsection \ref{subsec1.2}).

Our main purpose in this paper is to obtain the following type\footnote{Here, $L_{x,t}^2(w)$ is the weighted space $L^2(wdxdt)$ and
$\|f\|_{\dot{H}^s}:=\|(-\Delta)^{\frac s2}f\|_{L^2}$ is the homogeneous Sobolev norm of order $s\in\mathbb{R}$.} of weighted estimates
\begin{equation}\label{stsm}
\|e^{-it(-\Delta)^{\gamma/2}}f\|_{L_{x,t}^2(w(x,t))}\leq C\|w\|_{\mathfrak{L}^{2s+\gamma,p}_\gamma}^{1/2}\|f\|_{\dot{H}^s},
\end{equation}
where $\mathfrak{L}^{\alpha,p}_\gamma$ is a function class of weights $w\geq0$ equipped with the norm
$$\|w\|_{\mathfrak{L}^{\alpha,p}_\gamma}
:=\sup_{(x,t)\in\mathbb{R}^{n+1},r>0}r^\alpha
\bigg( \frac{1}{r^{n+\gamma}} \int_{Q(x,r) \times I(t,r^\gamma)} w(y,s)^p dyds \bigg)^{1/p}$$
for $\gamma\geq1$, $\alpha>0$ and $1\leq p\leq\frac{n+\gamma}\alpha$.
Here, $Q(x,r)$ denotes a cube in $\mathbb{R}^n$ centered at $x$ with side length $r$,
and $I(t,r^\gamma)$ denotes an interval in $\mathbb{R}$ centered at $t$ with length $r^\gamma$.

Notice that $\mathfrak{L}^{\alpha,p}_\gamma$ when $\gamma=1$ is just the usual Morrey-Campanato class $\mathfrak{L}^{\alpha,p}$ on $\mathbb{R}^{n+1}$,
and its norm has the following homogeneity
$$\|f(\lambda x,\lambda^\gamma t)\|_{\mathfrak{L}^{\alpha,p}_\gamma}
=\lambda^{-\alpha}\|f\|_{\mathfrak{L}^{\alpha,p}_\gamma}.$$
This observation motivates the definition of the class $\mathfrak{L}^{\alpha,p}_\gamma$.
In other words, they are anisotropic variants of the Morrey-Campanato class adapted to
scaling considerations $(x,t)\mapsto(\lambda x,\lambda^\gamma t)$ in regard to equation \eqref{disp}.
In this regard, we shall call $\mathfrak{L}^{\alpha,p}_\gamma$ $\gamma$-order anisotropic Morrey-Campanato class.
It should be now noted that $\alpha=2s+\gamma$ in \eqref{stsm} is automatically determined from the scaling invariance.
Notice also that $\mathfrak{L}^{\alpha,p}_\gamma =L^p$ when $p=\frac{n+\gamma}\alpha$,
and even $L^{p,\infty}\subset\mathfrak{L}^{\alpha,p}_\gamma$ for $p<\frac{n+\gamma}\alpha$.
So, $\mathfrak{L}^{\alpha,p}_\gamma$ can be seen as natural extensions of $L^p$ class, which contains more singular functions
like $|(x,t)|^{-(n+1)/p}$ for $p<\frac{n+\gamma}\alpha$ which do not belong to any $L^p$ class.
Finally, $\mathfrak{L}^{\alpha,p}_\gamma \subset \mathfrak{L}^{\alpha,q}_\gamma$ for $q<p$.

Estimates \eqref{stsm} when $s<0$ are indicative of smoothing estimates
\begin{equation*}
\big\||\nabla|^{-s}e^{-it(-\Delta)^{\gamma/2}}f\big\|_{L_{x,t}^2(w(x,t))} \leq C\|w\|_{\mathfrak{L}^{2s+\gamma,p}_\gamma}^{1/2} \|f\|_{L^2}.
\end{equation*}
On the other hand, they may be thought of as weighted Strichartz estimates in case of $s\geq0$,
which are more adapted than the classical ones to the study of the well-posedness theory for relevant equations in weighted spaces.
In fact, we will also obtain some corresponding inhomogeneous estimates (see Propositions \ref{prop}, \ref{prop00} and \ref{prop010}) and apply
them together with \eqref{stsm} to the well-posedness for various equations
like Schr\"odinger and wave equations with time-dependent potentials.
This is our secondary objective in this paper.

\subsection{Strichartz and smoothing estimates}
From now on, we will describe our results in detail.
We start with the Schr\"odinger equation that is the case $\gamma=2$ in \eqref{disp}.
It describes the wave behavior of a quantum particle
that is moving in the absence of external forces.
The physical interpretation for its solution is that
$|u(x,t)|^2$ is the probability density for finding the particle at place $x\in\mathbb{R}^n$ and time $t\in\mathbb{R}$.
This is led to think that $L^2(\mathbb{R}^n)$ will play a distinguished role.
In fact, by Plancherel's theorem, $e^{it\Delta}$ defines an isometry on $L_x^2$.
That is,
$\|e^{it\Delta}f\|_{L_x^2}=\|f\|_{L^2}$ for any fixed $t$.
But interestingly, if we take $L^2$-norm in $t$ in the case $n=1$,
the extra gain of regularity of order $1/2$ in $x$ can be observed. That is,
for any fixed $x\in\mathbb{R}$,
$$\|e^{it\Delta}f\|_{L_t^2}=C\|f\|_{\dot{H}^{-\frac12}},$$
which follows again from Plancherel's theorem.

This kind of smoothing effect for general space dimensions
is known to be true in a weighted $L_{x,t}^2$ space with a singular power weight in the spatial variable.
In fact,
\begin{equation}\label{KY}
\|e^{it\Delta}f\|_{L_{x,t}^2(|x|^{-2(s+1)})}\leq C_s\|f\|_{\dot{H}^s}
\end{equation}
if and only if $-\frac12<s<\frac{n-2}2$ and $n\geq2$.\footnote{This estimate was first discovered by Kato and Yagima \cite{KY}
for $-\frac12<s\leq0$ whenever $n\geq3$, and $-\frac12<s<0$ for $n=2$,
and an alternate proof of this result was given by Ben-Artzi and Klainerman \cite{BK}.
After that, it was shown by Watanabe \cite{Wa} that \eqref{KY} fails for the case $s=-\frac12$,
and the full range was obtained by Sugimoto \cite{Su} although it was later shown by Vilela \cite{V} that the range is indeed optimal.}
Now, by observing that $|x|^{-\alpha}\in\mathfrak{L}^{\alpha,p}(\mathbb{R}^n)$, $p<n/\alpha$, it seems natural to expect
a more general estimate,
\begin{equation}\label{indepen}
\|e^{it\Delta}f\|_{L_{x,t}^2(w(x))}\leq C\|w\|_{\mathfrak{L}^{2s+2,p}}^{1/2}\|f\|_{\dot{H}^s}
\end{equation}
for some $-1<s\leq\frac{n-2}2$ and $1\leq p\leq \frac n{2s+2}$.
Indeed, Barcel\'{o} et al \cite{BBCRV} obtained the following equivalent estimates
\begin{equation}\label{essenindepen}
\|e^{it\Delta}f\|_{L_{x,t}^2(w(x,t))}\leq C\|\sup_{t\in\mathbb{R}}w(x,t)\|_{\mathfrak{L}^{2s+2,p}}^{1/2}\|f\|_{\dot{H}^s}
\end{equation}
for $-\frac1{n+1}<s\leq\frac{n-2}2$ and $\frac{n-1}{2(2s+1)}<p\leq\frac n{2s+2}$ with $n\geq2$
(see also \cite{RV2} for $s=0$).
They also obtained some necessary conditions for \eqref{essenindepen} which particularly shows that their result is sharp when $s<0$ except for the border line $p=\frac{n-1}{2(2s+1)}$.
In other words, it is not possible to gain a regularity of order $s>\frac1{n+1}$ in the setting \eqref{indepen}.

Inspired by \eqref{indepen} (or \eqref{essenindepen}), we may now expect more subtle estimates,
\begin{equation}\label{depen}
\|e^{it\Delta}f\|_{L_{x,t}^2(w(x,t))}\leq C\|w\|_{\mathfrak{L}_2^{2s+2,p}}^{1/2}\|f\|_{\dot{H}^s}
\end{equation}
for some $-1<s\leq \frac n2$ and $1\leq p\leq\frac{n+2}{2s+2}$.
In fact, since $\mathfrak{L}_2^{\alpha,p}(\mathbb{R}^{n+1})$ is an anisotropic variant of $\mathfrak{L}^{\alpha,p}(\mathbb{R}^n)$
and $\mathfrak{L}_x^{\alpha,p}L_t^\infty\subset\mathfrak{L}_2^{\alpha,p}(\mathbb{R}^{n+1})$,
estimates \eqref{depen} can be viewed as natural extensions of \eqref{indepen} (or \eqref{essenindepen}) to more general time-dependent weights.
Recently, there was an attempt \cite{BBCRV2} to obtain \eqref{depen} in the case of $s\geq0$ only.
This work was inspired by
\begin{equation}\label{wstr}
\|e^{it\Delta}f\|_{L_{x,t}^2(w(x,t))}\leq C\|w\|_{L_{x,t}^p}^{1/2}\|f\|_{\dot{H}^s}
\end{equation}
where $0\leq s<\frac n2$ and $p=\frac{n+2}{2s+2}$,
which follows from H\"older's inequality along with the classical Strichartz estimates\footnote{These can be obtained combining the Hardy-Littlewood-Sobolev inequality
and the mixed-norm Strichartz estimates
$\|e^{it\Delta}f\|_{L_t^qL_x^r}\leq C\|f\|_{L^2}$
which holds if and only if
$\frac2q+\frac nr=\frac n2$, $q,r\geq2$ and $(q,r,n)\neq(2,\infty,2)$ (see \cite{Str2,GV,M,KT}).}
$\|e^{it\Delta}f\|_{L_{x,t}^{\widetilde{r}}}\leq C\|f\|_{\dot{H}^s}$
where $0\leq s<\frac n2$ and $\widetilde{r}=\frac{2(n+2)}{n-2s}$.
In fact, since $L^p(\mathbb{R}^{n+1})=\mathfrak{L}^{2s+2,p}_2(\mathbb{R}^{n+1})$ for $p=\frac{n+2}{2s+2}$,
estimate \eqref{wstr} gives \eqref{depen} at least on the critical line $p=\frac{n+2}{2s+2}$, the segment $[E,C)$,
and then opens up the possibility of \eqref{depen} off this trivial line,
thus giving improvements on \eqref{wstr}.
See Figure 1.

It was shown by Barcel\'{o} et al \cite{BBCRV2}
that \eqref{depen} do indeed hold at least when $s$ is sufficiently large.
Precisely, they showed that \eqref{depen} holds if $\frac n4\leq s<\frac n2$ and $1<p\leq\frac{n+2}{2s+2}$,
i.e., if $(s,1/p)$ lies in the closed triangle with vertices $B,D,C$ except for the segment $[B,C]$,
and is false if $0\leq s<\frac n4$ and $p<\frac{n+4}{4s+4}$,
i.e., if $(s,1/p)$ lies in the closed triangle with vertices $A,F,B$ except for the segment $[B,F]$.
And then they left unanswered the more difficult question of whether \eqref{depen} might hold
for $0\leq s<\frac n4$ and some $\frac{n+4}{4s+4}\leq p<\frac{n+2}{2s+2}$, i.e., for some $(s,1/p)$ which lies
in the closed quadrangle with vertices $B,D,E,F$ except for the segments $[B,D]$ and $[D,E]$.

\begin{figure}[t!]
\includegraphics[width=11.0cm]{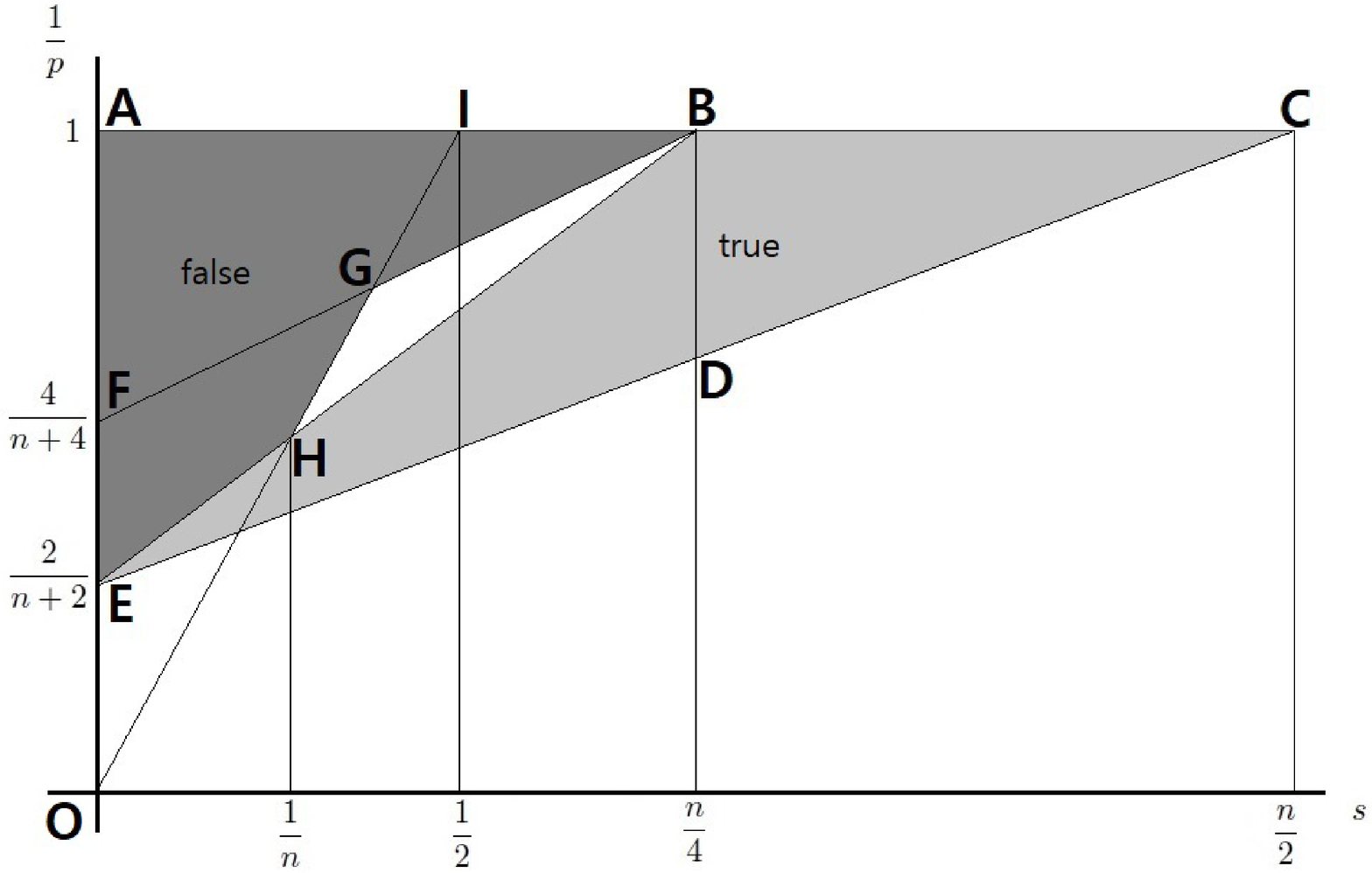}
\caption{The region of $(s,1/p)$ for the Schr\"odinger equation when $n\geq3$}
\end{figure}

Here we answer this open question as a consequence of Theorem \ref{thm} and Proposition \ref{prop0} below.
In fact, the theorem particularly shows that there exist some $p<\frac{n+2}{2s+2}$
for which \eqref{depen} holds for $0<s<\frac n4$ (see Remark \ref{rem0}),
and from the proposition one can see that there can be no such $p<\frac{n+2}{2s+2}$ when $s=0$ (see Remark \ref{rem}).

\begin{thm}\label{thm}
Let $n\geq1$. Then we have
\begin{equation}\label{depen2}
\|e^{it\Delta}f\|_{L_{x,t}^2(w(x,t))}\leq C\|w\|_{\mathfrak{L}^{2s+2,p}_2}^{1/2}\|f\|_{\dot{H}^s}
\end{equation}
if $0<s<\frac n2$ and\, $\max\{1,\frac{n+2}{4s+2}\}<p\leq\frac{n+2}{2s+2}$.
\end{thm}

\begin{rem}\label{rem0}
This theorem entirely recovers the previous result mentioned above because $\max\{1,\frac{n+2}{4s+2}\}=1$ for $s\geq\frac n4$,
and particularly shows that \eqref{depen2} can hold if $0<s<\frac n4$ and $\frac{n+2}{4s+2}<p\leq\frac{n+2}{2s+2}$,
i.e., if $(s,1/p)$ lies in the closed triangle with vertices $B,E,D$ except for the segment $[B,E]$.
\end{rem}

Notice also that the theorem is sharp for $s\geq\frac n4$ except for the border line $p=1$, the segment $[B,C]$.
We shall continue to discuss the sharpness of the theorem for $s<\frac n4$ through the following proposition which
gives a new necessary condition on $p$ and $s$ for \eqref{depen2} to hold.
For details, see remarks below the proposition.

\begin{prop}\label{prop0}
Let $-1<s<n/4$. Estimate \eqref{depen2} is false if
\begin{equation*}
\frac1p>\frac{4s+2}{n+2}\quad\text{and}\quad \frac1p>2s.
\end{equation*}
\end{prop}

\begin{rem}\label{rem}
The lines $BE$ and $IO$ in Figure 1 correspond to $\frac1p=\frac{4s+2}{n+2}$ and $\frac1p=2s$, respectively.
Compared with the previously known result, the proposition therefore gives a new region of $(s,1/p)$ on which \eqref{depen2} is false
when $0\leq s<\frac n4$.
This region consists of $(s,1/p)$ which lies in
the closed quadrangle with vertices $E,F,G,H$ except for the segments $[E,H]$ and $[G,H]$.
\end{rem}

\begin{rem}
Since the line $BE$ is above $IO$ whenever $n\leq2$ or $s\leq1/n$,
we see that Theorem \ref{thm} is sharp when $n\leq2$ except for the border lines $[B,E)$ and $[B,C]$,
and is sharp for $s\leq1/n$ when $n\geq3$ except for the border line $[H,E)$.
It is now plausible to conjecture that the line $BE$ is the border line for which \eqref{depen2} holds.
But, if one considers radial initial data $f$, then one can see, from our previous result \cite{KoS3} (see Remark 1.3 there),
that \eqref{depen2} can hold beyond this border line.
\end{rem}

Beyond the case $s\geq0$, we further ask how much regularity we can expect on \eqref{depen2}.
But interestingly, since $\frac1p\geq\frac{2s+2}{n+2}>\frac{4s+2}{n+2}>2s$ for $s<0$,
from Proposition \ref{prop0} we see that \eqref{depen2} is false whenever $s<0$.
In other words, the smoothing effect cannot occur in the time-dependent setting \eqref{depen2}.
This shows that the dispersion in the Schr\"odinger equation is not strong enough to have the smoothing effect
and naturally leads us to consider the possibility of having the effect at best
in higher-order cases $\gamma>2$ in \eqref{disp} whose dispersion is more strong.
In the following theorem we obtain a smoothing effect exactly for these higher-order cases.
But, if we consider radial initial data $f$, there is still such a possibility in lower-order cases where $1< \gamma \leq2$.
See the first paragraph below Remark \ref{rem2}.

\begin{thm}\label{thm2}
Let $n\geq1$ and $\gamma>1$.
Then we have
\begin{equation}\label{higher}
\|e^{it(-\Delta)^{\gamma/2}}f\|_{L^2(w(x,t))} \leq C\|w\|_{\mathfrak{L}^{2s+\gamma,p}_{\gamma}}^{1/2} \|f\|_{\dot{H}^s}
\end{equation}
if $-\frac{(\gamma-2)n}{2(n+2)} <s< \frac n2$ and $\max\{1,\frac{n+2(\gamma-1)}{4s+2(\gamma-1)}\} <p\leq \frac{n+\gamma}{2s+\gamma}$.
\end{thm}

\begin{rem}
The case $\gamma=2$ in this theorem is, of course, reduced to Theorem \ref{thm}.
When $\gamma>2$, we have a smoothing effect in \eqref{higher} with a gain of regularity of order $s<\frac{(\gamma-2)n}{2(n+2)}$.
Some angular smoothing results are known for $w(x,t)=|x|^{-(2s+\gamma)}$ (see \cite{FW} and reference therein).
Note that $|x|^{-(2s+\gamma)}\in\mathfrak{L}^{2s+\gamma,p}_{\gamma}$.
Hence our smoothing results extend these previous ones to more general time-dependent weights
even without assuming any angular regularity on the initial data $f$.
\end{rem}

\begin{rem}\label{rem2}
In \cite{KoS2}, we have implicitly obtained \eqref{higher} for large $\gamma>(n+2)/2$ with $\mathfrak{L}^{2s+\gamma,p}_{\gamma}$
replaced by a different Morrey-Campanato type class\footnote{Notice that
$\mathfrak{L}^{\beta,p}(\mathbb{R};\mathfrak{L}^{\alpha,p}(\mathbb{R}^n))\subset\mathfrak{L}^{\alpha,\beta,p}$.}, $\mathfrak{L}^{\alpha,\beta,p}$, equipped with the norm
\begin{equation}\label{diff}
\|w\|_{\mathfrak{L}^{\alpha,\beta,p}}
:=\sup_{(x,t)\in\mathbb{R}^{n+1},r,l>0}r^\alpha l^\beta
\bigg(\frac{1}{r^nl} \int_{Q(x,r) \times I(t,l)} w(y,s)^p dyds \bigg)^{\frac{1}{p}}
\end{equation}
which is defined for $0<\alpha\leq \frac np$, $0<\beta\leq\frac1p$ and $p\geq1$.
But, $\mathfrak{L}^{\alpha,\beta,p}\subset\mathfrak{L}^{\alpha+ \gamma\beta, p}_\gamma$ by taking $r=l^{1/\gamma}$ in \eqref{diff}.
By the scaling, the condition $\alpha+ \gamma\beta = 2s+\gamma$ is also assumed to be valid for which the estimate holds in the case of $\mathfrak{L}^{\alpha,\beta,p}$ replacing  $\mathfrak{L}^{2s+\gamma,p}_{\gamma}$.
Hence the present work significantly improves this previous one.
\end{rem}

The smoothing effect can occur when $1< \gamma \leq2$ as well, if we consider radial $f$.
In fact, we recently showed for radial $f$ that \eqref{higher} holds if
$-\frac{(\gamma-1)n -\gamma}{2n} <s< \frac n2$ and $\max\{1,\frac{\gamma}{2s+\gamma-1}\} <p\leq \frac{n+\gamma}{2s+\gamma}$ whenever $n\geq2$ and $\gamma>1$ (see Remark 1.3 in \cite{KoS3}).
This shows that we can have a smoothing effect in the radial case with a gain of regularity of order $s<\frac{(\gamma-1)n-\gamma}{2n}$
whenever $n>\frac{\gamma}{\gamma-1}$ and $\gamma>1$.

\begin{figure}[t!]
\includegraphics[width=11.0cm]{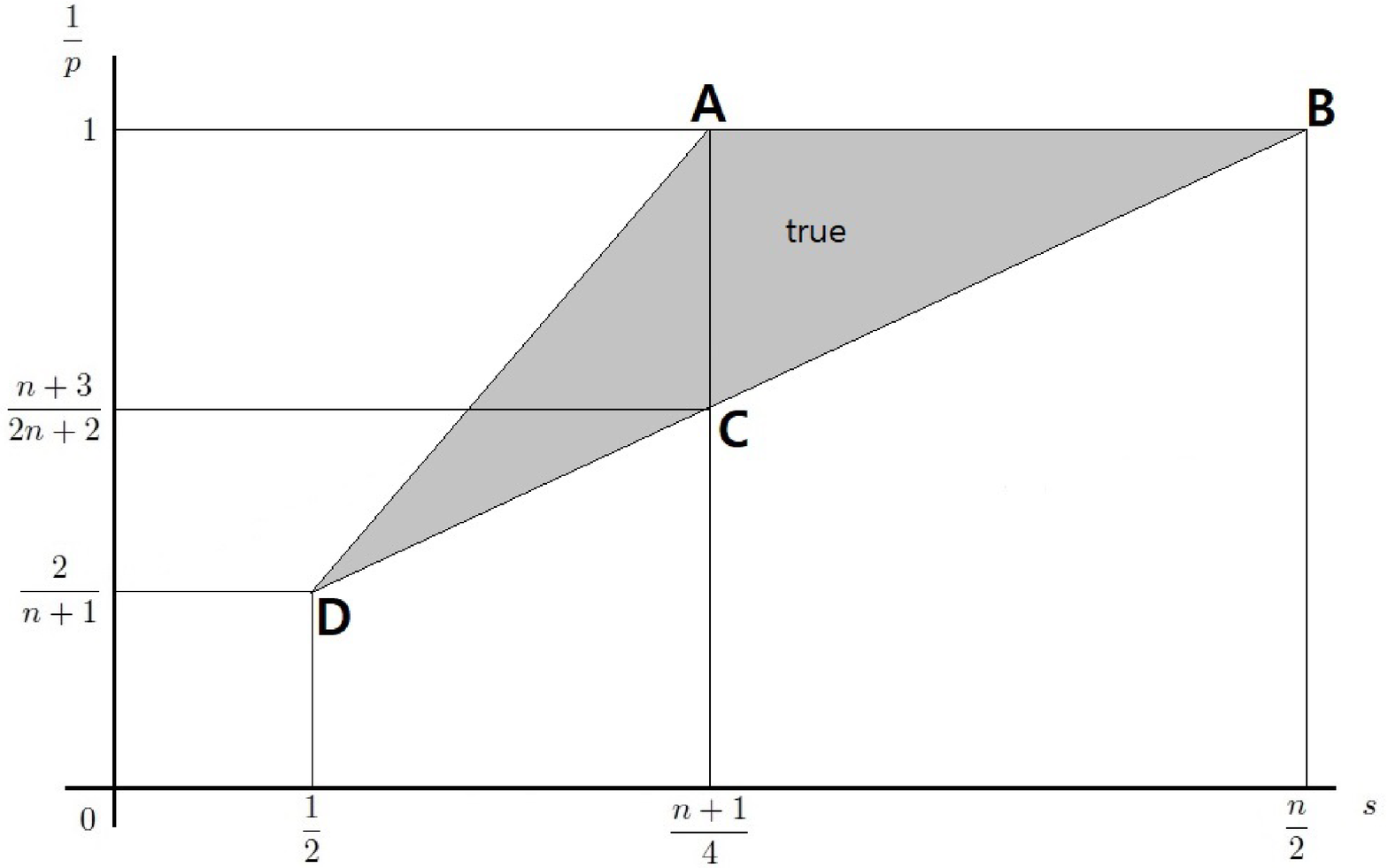}
\caption{The region of $(s,1/p)$ for the wave equation when $n\geq2$}
\end{figure}

Compared to \eqref{higher}, the following estimates for the case $\gamma=1$ are somewhat inferior but this is due to the weaker dispersion in the wave equation.

\begin{thm}\label{thm3}
Let $n\geq2$.
Then we have
\begin{equation}\label{wave}
\big\|e^{it\sqrt{-\Delta}}f \big\|_{L^2(w(x,t))}\leq C\|w\|_{\mathfrak{L}_1^{2s+1,p}}^{1/2}\|f\|_{\dot{H}^s}
\end{equation}
if $\frac12<s<\frac n2$ and $\max\{1,\frac{n+1}{4s}\}<p\leq\frac{n+1}{2s+1}$.
\end{thm}

\begin{rem}
In \cite{KoS}, we have already obtained \eqref{wave} when $\frac{n+1}{4}\leq s<\frac n2$ and $1<p\leq\frac{n+1}{2s+1}$,
i.e., when $(s,1/p)$ lies in the closed triangle with vertices $A,B,C$ except for the segment $[A,B]$.
Since $\max\{1,\frac{n+1}{4s}\}=1$ for $s\geq\frac{n+1}{4}$, the theorem entirely recovers this result,
and significantly extends it to the case when
$\frac12<s<\frac{n+1}{4}$ with $\frac{n+1}{4s}<p\leq\frac{n+1}{2s+1}$,
i.e., when $(s,1/p)$ lies in the closed triangle with vertices $A,C,D$ except for the segment $[A,D]$.
See Figure 2.
\end{rem}

\begin{rem}
From the classical Strichartz's estimate (\cite{Str2}), one can see that
\begin{equation*}
\big\|e^{it\sqrt{-\triangle}}f\big\|_{L^r_{x,t}}
\leq C\|f\|_{\dot{H}^s}
\end{equation*}
for $2(n+1)/(n-1)\leq r<\infty$ and $s=n/2-(n+1)/r$.
By H\"older's inequality along with this estimate we get
\begin{equation}\label{fcs}
\big\|e^{it\sqrt{-\triangle}}f\big\|_{L^2(w(x,t))}^2
\leq\|w\|_{L^{\frac{r}{r-2}}}\|e^{it\sqrt{-\triangle}}f\|_{L^r}^2
\leq C\|w\|_{L^{\frac{n+1}{2s+1}}}\|f\|_{\dot{H}^s}^2
\end{equation}
for $1/2\leq s<n/2$.
Since $L^p(\mathbb{R}^{n+1})=\mathfrak{L}^{2s+1,p}_1(\mathbb{R}^{n+1})$ for $p=\frac{n+1}{2s+1}$,
this estimate gives \eqref{wave} on the critical line $p=\frac{n+2}{2s+2}$, the segment $[D,B)$.
Hence \eqref{wave} can be seen as natural extensions to
the Morrey-Campanato classes of \eqref{fcs}.
\end{rem}

Some weighted $L_{x,t}^2$ estimates are known for \eqref{wave}, but with time-independent weights $w(x)$.
In particular, it can be found in \cite{RV2} for $w(x)\in\mathfrak{L}^{2,p}(\mathbb{R}^n)$
with $p>(n-1)/2$, $n\geq3$.
For a more specific weight $w(x)=|x|^{-(2s+1)}$ with $0<s<(n-1)/2$, see (3.6) in \cite{HMSSZ}.
Note again that $|x|^{-(2s+1)}\in\mathfrak{L}_1^{2s+1,p}(\mathbb{R}^{n+1})$.
Hence, \eqref{wave} gives estimates for more general time-dependent weights $w(x,t)$.

\subsection{Applications}\label{subsec1.2}
Now we turn to a few applications of our estimates to global well-posedness of the Cauchy problem for dispersive equations
in which the most fundamental differential operators
may be Schr\"odinger ($i\partial_t+\Delta$) and wave ($\partial_t^2-\Delta$) operators
that are second-order.
The higher-order counterparts of them have been also considerably attracted in recent years from mathematical physics.
For example, the fourth-order Schr\"odinger operator $i\partial_t-\Delta^2$
was introduced in \cite{K,K2,KS} to consider the role of small fourth-order dispersion
terms in the propagation of intense laser beams in a bulk medium.
The fourth-order wave operator $\partial_t^2+\Delta^2$ has been involved in the study of plate and beams (\cite{Lo}).

In this regard, we shall focus our attention on the following Cauchy problems,
\begin{equation}\label{Welsch}
\begin{cases}
i\partial_tu-(-\Delta)^{\gamma/2} u +V(x,t)u=F(x,t),\\
u(x,0)=f(x)
\end{cases}
\end{equation}
and
\begin{equation}\label{Welwav}
\begin{cases}
\partial_t^2u+(-\Delta)^{\gamma/2}u +V(x,t)u=F(x,t),\\
u(x,0)=f(x),\\
\partial_tu(x,0)=g(x)
\end{cases}
\end{equation}
with $\gamma\geq2$.
Our second aim in this paper is then to find a suitable condition on the perturbed term $V(x,t)$
which guarantees that these problems are globally well-posed in the weighted $L^2$ space, $L^2(|V|dxdt)$.
See Section \ref{sec6} for details.
Our method may have further applications for other dispersive equations.
It can be applied to the linearized KdV type equations,
$\partial_tu+\partial_x^{2k+1}u+V(x,t)u=0$, $u(x,0)=f(x)$.
It takes up the final subsection \ref{subsec7.2}.

A natural way to achieve this well-posedness is to control weighted $L^2$ integrability of the solution
in terms of regularity of the Cauchy data and the forcing term.
To be precise, let us first observe that
the solutions to \eqref{Welsch} and \eqref{Welwav} are given by
\begin{equation}\label{solsch}
u(x,t)=e^{-it(-\Delta)^{\gamma/2}}f(x) -i\int_0^t e^{-i(t-s)(-\Delta)^{\gamma/2}}(F-Vu)(\cdot,s)ds
\end{equation}
and
\begin{equation}\label{solwave}
\begin{aligned}
u(x,t)=\Re\bigg[e^{-it\sqrt{(-\Delta)^{\gamma/2}}}\bigg]f(x) &+i\Im\bigg[\frac{e^{-it\sqrt{(-\Delta)^{\gamma/2}}}}{\sqrt{(-\Delta)^{\gamma/2}}}\bigg]g(x)\\
-i\int_0^t&\Im\bigg[\frac{e^{-i(t-s)\sqrt{(-\Delta)^{\gamma/2}}}}{\sqrt{(-\Delta)^{\gamma/2}}}\bigg] (F-Vu)(\cdot,s)ds,
\end{aligned}
\end{equation}
respectively,
via the Fourier transform and Duhamel's principle.
Here $\Re$ and $\Im$ denote the real and imaginary parts of the relevant operators.
Our basic strategy is then based on obtaining suitable weighted-$L^2$ Strichartz estimates for the homogeneous and inhomogeneous components
of the solutions \eqref{solsch} and \eqref{solwave}.
For the homogeneous components we will use the estimates in Theorems \ref{thm2} and \ref{thm3}.
On the other hand, for the inhomogeneous components
we will obtain the following estimates (Propositions \ref{prop}, \ref{prop00} and \ref{prop010})
which are adapted to obtaining our well-posedness results in the weighted setting.

\begin{prop}\label{prop}
Let $n\geq1$ and $\gamma\geq2$.
Then we have
\begin{equation}\label{higher3}
\bigg\|\int_{0}^{t}e^{-i(t-s)(-\Delta)^{\gamma/2}}F(\cdot,s)ds\bigg\|_{L^2(w(x,t))}
\leq C\|w\|_{\mathfrak{L}^{\gamma,p}_{\gamma}} \|F\|_{L^2(w(x,t)^{-1})}
\end{equation}
if $\max\{1,\frac{n+2(\gamma-1)}{2(\gamma-1)}\}<p\leq\frac{n+\gamma}{\gamma}$ when $\gamma>2$, and if $p=\frac{n+2}{2}$ when $\gamma=2$.
\end{prop}

Recall here that $p=\frac{n+2}{2}$ when $\gamma=2$ is the best possible for which the homogeneous estimate \eqref{depen2}
holds for $s=0$. In view of the well-posedness for $L^2$ initial data, $p=\frac{n+2}{2}$ is therefore the correct one
for the corresponding inhomogeneous part \eqref{higher3} when $\gamma=2$.
Compared with \eqref{higher3}, we similarly have the following estimate concerning wave equations.

\begin{prop}\label{prop00}
Let $n\geq2$ and $2\leq\gamma<2n$.
Then we have
\begin{equation}\label{121}
\bigg\|\int_{0}^{t} \frac{ e^{i(t-s)\sqrt{(-\Delta)^{\gamma/2}}} }{\sqrt{(-\Delta)^{\gamma/2}}} F(\cdot,s)ds\bigg\|_{L^2(w(x,t))}
\leq C \|w\|_{\mathfrak{L}^{\gamma,p}_{\gamma/2}} \|F\|_{L^2(w(x,t)^{-1})}
\end{equation}
if $\max\{1,\frac{n-2+\gamma}{2(\gamma-1)}\}<p\leq\frac{2n+\gamma}{2\gamma}$ when $\gamma>2$, and if $p=\frac{n+1}{2}$ when $\gamma=2$.
\end{prop}

Our approach to these inhomogeneous estimates is similar to the one for the homogeneous case.
But we shall also adopt a different approach based on the fractional integral
to improve the restriction $\gamma<2n$ in the case of \eqref{121} to $\gamma<3n$
(see Remark \ref{impp}).
The resulting estimates are given as follows:

\begin{prop}\label{prop010}
Let $n\geq1$ and $2\leq\gamma<3n$.
If $-\frac{(\gamma-4)n}{4(n+2)}<s<\frac12\min\{\gamma,n\}$, $1<r\leq \frac{2n}{\gamma-2s}$
and $\max\{1,\frac{n+\gamma-2}{4s+\gamma-2}\} <p\leq \frac{2n+\gamma}{4s+\gamma}$ when $2<\gamma<2n+2s$, then we have
\begin{equation}\label{fracver}
\bigg\|\int_{0}^{t}\frac{e^{i(t-s)\sqrt{(-\Delta)^{\gamma/2}}}}{\sqrt{(-\Delta)^{\gamma/2}}}F(\cdot,s)ds\bigg\|_{L^2(w(x,t))}
\leq C\|w\|_{L_t^1\mathfrak{L}^{\frac\gamma2-s,r}}^{1/2}\|w\|_{\mathfrak{L}^{2s+\gamma/2,p}_\gamma}^{1/2}\|F\|_{L^2(w(x,t)^{-1})}.
\end{equation}
When $\gamma=2$ and $n\geq2$, this holds if $\frac12<s<1$, $1<r\leq \frac{n}{1-s}$ and $\max\{1, \frac{n+1}{4s} \} <p \leq \frac{n+1}{2s+1}$.
\end{prop}

\begin{rem}\label{impp}
Given $2\leq\gamma<3n$, one can easily check that there are $s,r,p,\gamma$ satisfying the conditions in the proposition.
\end{rem}

\subsection{Main ideas}

We shall now give an outline of the main ideas in our work.

The known approach to Strichartz and smoothing estimates with time-independent weights
is based on weighted $L^2$ bounds for the resolvent operator as well as for the Fourier restriction
(see, for example, \cite{RV2,Su,BBCRV,BBRV,S2}),
but it is no longer available in the case of time-dependent weights.
Our method that works for this case is completely different from it and is a more fruitful approach
which is based on a combination between two kind of arguments, one from the theory of dispersive estimates
and the other one from weighted inequalities.
Even if these two kind of ideas ($TT^*$ argument and bilinear approach for dispersive estimates in one hand
and maximal functions with $A_p$ class for weighted estimates) are relatively well-known, this combination
seems original and new.
This is done in several steps:

\smallskip

\noindent\textit{Frequency localization based on maximal functions of weights.}
To obtain our estimates in this paper, the first step is to work on spatial Fourier transform side
by using the Littlewood-Paley theorem on weighted $L^2$ spaces with Muckenhoupt $A_2$ weights.
A key observation in this step is to remove this $A_2$ assumption, $w(\cdot,t)\in A_2(\mathbb{R}^n)$,
when applying the theorem.
For this we make use of a useful property (Lemma \ref{lem2'}) of $n$-dimensional maximal functions
$w_*(x,t)=(M(w(\cdot,t)^q)(x))^{1/q}$ of Morrey-Campanato type weights $w$.
Here, $M(f)(x)$ denotes the usual Hardy-Littlewood maximal function of $f$.
Such a property says that
$\|w_*\|_{\mathfrak{L}^{\alpha,p}_{\gamma}}\leq C\|w\|_{\mathfrak{L}^{\alpha,p}_{\gamma}}$ 
and $w_\ast(\cdot,t)\in A_2(\mathbb{R}^n)$ uniformly in $t\in\mathbb{R}$ and $w$.
Now, since $w\leq w_\ast$ and $\|w_*\|_{\mathfrak{L}^{\alpha,p}_{\gamma}}\leq C\|w\|_{\mathfrak{L}^{\alpha,p}_{\gamma}}$,
it suffices to show the estimates replacing $w$ with $w_\ast$.
So one may assume $w(\cdot,t)\in A_2(\mathbb{R}^n)$ for simplicity and can now apply the Littlewood-Paley theorem
without assuming the $A_2$ weight.
This finally leads us to estimating a number of frequency-localized pieces like
    \begin{equation}\label{freq_par00}
    \big\|e^{it(-\Delta)^{\gamma/2}} P_k f \big\|_{L^2(w(x,t))}\leq C2^{k(\alpha -\gamma)/2}
    \|w\|_{\mathfrak{L}^{\alpha,p}_{\gamma}}^{1/2} \|f\|_{L^2}
    \end{equation}
in Proposition \ref{Prop_F_local},
where $P_k$ is the Littlewood-Paley projection.
This frequency localization approach
based on such a useful property of maximal functions of Morrey-Campanato type weights
is the key ingredient in our argument which allows us to take advantage of localization in Fourier transform side.
See Section \ref{sec3} for details.

\smallskip

\noindent\textit{$TT^*$ argument and bilinear interpolation.}
The next step is devoted to proving \eqref{freq_par00}
whose proof is based on a combination of the usual $TT^*$ argument
and the bilinear interpolation.
We shall give here a brief description of this step. See Section \ref{sec4} for details.
Notice first that we only need to show the case $k=0$ in \eqref{freq_par00} by the scaling.
By the $TT^*$ argument and duality we are then reduced to showing the bilinear form estimate \eqref{key_est_1},
    \begin{equation*}
    \bigg|\bigg\langle\int_{-\infty}^te^{i(t-s)(-\Delta)^{\gamma/2}}P_0^2F(\cdot,s)ds,G(x,t)\bigg\rangle_{L_{x,t}^2}\bigg|
    \leq C\|w\|_{\mathfrak{L}^{\alpha,p}_{\gamma}} \|F\|_{L^2(w^{-1})}\|G\|_{L^2(w^{-1})}.
    \end{equation*}
Next we decompose dyadically the inner product on $L_{x,t}^2$ in time to get
    \begin{align*}
    \nonumber\bigg\langle\int_{-\infty}^{t}&e^{i(t-s)(-\Delta)^{\gamma/2}}P_0^2F(\cdot,s)ds,G(x,t)\bigg\rangle_{L_{x,t}^2}\\
    =&\sum_{j\geq0}\int_{\mathbb{R}}\int_{t- I_j}\Big\langle e^{i(t-s)(-\Delta)^{\gamma/2}}P_0^2 F(\cdot,s),G(x,t)\Big\rangle_{L^2_x}dsdt
    :=\sum_{j\geq0}T_j(F,G),
    \end{align*}
where $t-I_j=(t-2^j, t-2^{j-1}]$ for $j\geq1$, and $t-I_0=(t-1,t]$.
In Proposition \ref{tlocal},
we then obtain two bounds for these time-localized pieces, $T(F,G)=\{T_j(F,G)\}_{j\geq0}$,
from $L^2\times L^2$ to $\ell_\infty^{s_0}$ in one hand
and from $L^2(w^{-p})\times L^2(w^{-p})$ to $\ell_\infty^{s_1}$,
with the operator norms $C$ and
$C\|w\|_{\mathfrak{L}^{\alpha,p}_{\gamma}}^p$, respectively.
Here, $s_0=-1$ and $s_1=s_1(\alpha,p,\gamma,n)$.
($s_1=-(n/2+\gamma-p\alpha)$ if, for example, $\gamma>1$.)
Hence the bilinear interpolation (Lemma \ref{inter}), with $\theta=1/p$, $q=\infty$ and $p_1=p_2=2$, implies
\begin{equation}\label{bibi2}
T:(L^2,L^2(w^{-p}))_{1/p,2}\times(L^2,L^2(w^{-p}))_{1/p,2}
\rightarrow(\ell_\infty^{s_0},\ell_\infty^{s_1})_{1/p,\infty}
\end{equation}
with the operator norm
$C\|w\|_{\mathfrak{L}^{\alpha,p}_{\gamma}}$ for $p>1$.
Finally, by applying the real interpolation space identities in Lemma \ref{id} to \eqref{bibi2},
one can get
$$T:L^2(w^{-1})\times L^2(w^{-1})\rightarrow\ell_\infty^{s}$$
with some $s(\alpha,p,\gamma,n)$
and the operator norm $C\|w\|_{\mathfrak{L}^{\alpha,p}_{\gamma}}$ for $p>1$,
which is
    $$
    \begin{aligned}
    \bigg|\int_{\mathbb{R}}\int_{t- I_j} \Big\langle e^{i(t-s)(-\Delta)^{\gamma/2}} &P_0^2 F(\cdot,s),G(x,t)\Big\rangle dsdt\bigg| \\
    \leq C&2^{-js(\alpha,p,\gamma,n)} \|w\|_{\mathfrak{L}^{\alpha,p}_{\gamma}} \|F\|_{L^2(w^{-1})}\|G\|_{L^2(w^{-1})}
    \end{aligned}
    $$
for $p>1$.
By summing this over $j\geq0$, the bilinear form estimate follows
under $p>1$ and $s(\alpha,p,\gamma,n)>0$ which determine the conditions on $\alpha,p,\gamma,n$ for which
our estimates hold.

\medskip

Finally, let us sketch the organization of the paper.
In Section \ref{sec2} we present some preliminary lemmas which are used for the proof of Theorems \ref{thm2} and \ref{thm3}
that are proved in Sections \ref{sec3} and \ref{sec4} as described above.
In Section \ref{sec5} we prove Propositions \ref{prop}, \ref{prop00} and \ref{prop010}
which are additionally needed to obtain our well-posedness results in Section \ref{sec6} for Schr\"odinger and wave equations.
The final section, Section \ref{sec7} is devoted to proving Proposition \ref{prop0}
and giving further applications to linearized KdV type equations.

Throughout this paper, the letter $C$ stands for a positive constant which may be different
at each occurrence. We also denote by $\langle f,g\rangle$ the usual inner product of $f,g$ on $L^2$,
and denote $A\lesssim B$ and $A\sim B$ to mean $A\leq CB$ and $CB\leq A\leq CB$, respectively,
with unspecified constants $C>0$.


\section{Preliminaries}\label{sec2}

In this section we present some preliminary lemmas which will be used in later sections for the proof of Theorems \ref{thm2} and \ref{thm3}.

\subsection{Real interpolation spaces}
Given two complex Banach spaces $A_0$ and $A_1$, for $0<\theta<1$ and $1\leq q\leq\infty$,
we denote by $(A_0,A_1)_{\theta,q}$
the real interpolation spaces equipped with the norms
$$\|a\|_{(A_0,A_1)_{\theta,\infty}}=\sup_{0<t<\infty}t^{-\theta}K(t,a)$$
and
$$\|a\|_{(A_0,A_1)_{\theta,q}}=\big(\int_0^\infty(t^{-\theta}K(t,a))^q\frac{dt}{t}\big)^{1/q},\quad 1\leq q<\infty,$$
where
$$K(t,a)=\inf_{a=a_0+a_1}\|a_0\|_{A_0}+t\|a_1\|_{A_1}$$
for $0<t<\infty$ and $a\in A_0+A_1$.
In particular, $(A_0,A_1)_{\theta,q}=A_0=A_1$ if $A_0=A_1$.
See ~\cite{BL,T} for details.

We recall here two existing results concerning these real interpolation spaces.
The first one is the following bilinear interpolation lemma
which is well-known (see \cite{BL}, Section 3.13, Exercise 5(a)).

\begin{lem}\label{inter}
For $i=0,1$, let $A_i,B_i,C_i$ be Banach spaces and let $T$ be a bilinear operator such that
$$T:A_0\times B_0\rightarrow C_0$$
and
$$T:A_1\times B_1\rightarrow C_1.$$
Let $0<\theta<1$. Then one has
$$T:(A_0,A_1)_{\theta,p_1}\times(B_0,B_1)_{\theta,p_2}\rightarrow(C_0,C_1)_{\theta,q}$$
if\, $1\leq q\leq\infty$ and $1/q=1/p_1+1/p_2-1$.
\end{lem}

For $s\in\mathbb{R}$ and $1\leq q\leq\infty$,
let $\ell^s_q$ denote the weighted sequence space with the norm
$$\|\{x_j\}_{j\geq0} \|_{\ell^s_q}=
\begin{cases}
\big(\sum_{j\geq0}2^{jsq}|x_j|^q\big)^{1/q}
\quad\text{if}\quad q\neq\infty,\\
\,\sup_{j\geq0}2^{js}|x_j|
\quad\text{if}\quad q=\infty.
\end{cases}$$
Then the second lemma concerns some useful identities of real interpolation spaces of weighted spaces
(see Theorems 5.4.1 and 5.6.1 in \cite{BL}):

\begin{lem}\label{id}
Let $0<\theta<1$. Then one has
$$( L^{2} (w_0), L^{2} (w_1) )_{\theta,2} = L^2(w),\quad w= w_0^{1-\theta} w_1^{\theta},$$
and for $1\leq q_0,q_1,q\leq\infty$ and  $s_0\neq s_1$,
$$(\ell^{s_0}_{q_0}, \ell^{s_1}_{q_1} )_{\theta, q}=\ell^s_q,\quad s= (1-\theta)s_0 + \theta s_1.$$
\end{lem}

\subsection{Maximal functions of Morrey-Campanato weights}
Let us first recall that a weight\footnote{\,It is a locally integrable function
which is allowed to be zero or infinite only on a set of Lebesgue measure zero.}
$w:\mathbb{R}^n\rightarrow[0,\infty]$
is said to be in the Muckenhoupt $A_2(\mathbb{R}^n)$ class if there is a constant $C_{A_2}$ such that
\begin{equation*}
\sup_{Q\text{ cubes in }\mathbb{R}^{n}}
\bigg(\frac1{|Q|}\int_{Q}w(x)dx\bigg)\bigg(\frac1{|Q|}\int_{Q}w(x)^{-1}dx\bigg)<C_{A_2}.
\end{equation*}
(For details, see, for example, \cite{G}.)
In the following lemma we present a useful property of weights
in $\gamma$-order anisotropic Morrey-Campanato classes $\mathfrak{L}_\gamma^{\alpha,p}$
regarding their $n$-dimensional maximal functions
$$w_*(x,t)=(M(w(\cdot,t)^q)(x))^{1/q},$$
where $M(f)(x)$ denotes the usual Hardy-Littlewood maximal function of $f$.

\begin{lem}\label{lem2'}
For a weight $w\in\mathfrak{L}^{\alpha,p}_{\gamma}$ on $\mathbb{R}^{n+1}$,
let $w_*(x,t)$ be the $n$-dimensional maximal function given by
$$w_*(x,t)=\sup_{Q'}\bigg(\frac{1}{|Q'|}\int_{Q'}w(y,t)^q dy\bigg)^{\frac{1}{q}},\quad q>1,$$
where $Q'$ denotes a cube in $\mathbb{R}^n$ with center $x$.
Then, if $\alpha>\gamma/p$ and $p>q$, we have
$$\|w_*\|_{\mathfrak{L}^{\alpha,p}_{\gamma}}\leq C\|w\|_{\mathfrak{L}^{\alpha,p}_{\gamma}}$$
and $w_\ast(\cdot,t)\in A_2(\mathbb{R}^n)$ with a constant $C_{A_2}$ which is uniform in almost every $t\in\mathbb{R}$
and independent of $w$.
\end{lem}

Similar properties for various Morrey-Campanato type classes as well as this lemma have been obtained in our previous results \cite{KoS,KoS2,KoS3}.
Such a property for $\mathfrak{L}^{\alpha,p}_{\gamma}$ when $\alpha=2$ and $\gamma=1$ has been earlier developed in \cite{CS} and used in \cite{S,S2} concerning unique continuation for stationary and non-stationary Schr\"odinger equations, respectively.
For this lemma we refer the reader to Lemma 2.1 in \cite{KoS3}.
Compared with the lemma in \cite{KoS3}, the only additional part in this lemma is that
$w_\ast(\cdot,t)\in A_2(\mathbb{R}^n)$ with a constant $C_{A_2}$ which is also independent of $w$.
But this follows easily from the proof of the lemma in \cite{KoS3},
since the constant $C_{A_1}$ in (21) of \cite{KoS3} is independent of $w$.

\section{Proof of Theorems \ref{thm2} and \ref{thm3}}\label{sec3}

Let us first consider the multiplier operators $P_kf$ for $k\in\mathbb{Z}$ which are defined by
$$\widehat{P_kf}(\xi)=\phi(2^{-k}|\xi|)\widehat{f}(\xi),$$
where $\phi:\mathbb{R}\rightarrow[0,1]$ is a smooth cut-off function supported in $(1/2,2)$ such that
$$\sum_{k=-\infty}^\infty \phi(2^k t)=1,\quad t>0.$$
Then we will obtain the following frequency localized estimates in the next section and in this section we shall show that
these estimates imply Theorems \ref{thm2} and \ref{thm3} by combining Lemma \ref{lem2'} and the Littlewood-Paley theorem on weighted $L^2$ spaces.

\begin{prop}\label{Prop_F_local}
Let $n\geq1$ and $\gamma\geq1$.
Then we have
	\begin{equation}\label{F_local_homo}
	\big\|e^{it(-\Delta)^{\gamma/2}} P_k f \big\|_{L^2(w(x,t))}\leq C2^{k(\alpha -\gamma)/2}
	\|w\|_{\mathfrak{L}^{\alpha,p}_{\gamma}}^{1/2} \|f\|_{L^2}
	\end{equation}
if\, $p>1$ and
	\begin{equation*}
	\begin{aligned}
	\begin{cases}
	\alpha>1+\frac{n-2+2\gamma}{2p}
	\quad\text{when}\quad \gamma>1,\\
	\alpha>1+\frac{n+1}{2p}
	\quad\text{when}\quad\gamma=1.
	\end{cases}
	\end{aligned}
	\end{equation*}
\end{prop}

To deduce Theorems \ref{thm2} and \ref{thm3} from this proposition, we first observe that we may assume
$w(\cdot,t)\in A_2(\mathbb{R}^n)$ with $C_{A_2}$ which is uniform in almost every $t\in\mathbb{R}$
and independent of $w$.
Indeed, since $w\leq w_\ast$ and
$\|w_*\|_{\mathfrak{L}^{\alpha,p}_{\gamma}}\leq C\|w\|_{\mathfrak{L}^{\alpha,p}_{\gamma}}$ for $\alpha>\gamma/p$ and $p>1$
(see Lemma \ref{lem2'}), if we show the estimates in the theorems replacing $w$ with $w_\ast$, we get for $\gamma\geq1$
    \begin{align*}
    \big\|e^{it(-\Delta)^{\gamma/2}} f \big\|_{L^2(w(x,t))}
    &\leq \big\|e^{it(-\Delta)^{\gamma/2}} f \big\|_{L^2(w_*(x,t))} \\
    &\leq C\|w_*\|_{\mathfrak{L}^{\alpha,p}_{\gamma}}^{1/2}\|f\|_{\dot{H}^s}\\
    &\leq C\|w\|_{\mathfrak{L}^{\alpha,p}_{\gamma}}^{1/2}\|f\|_{\dot{H}^s}
    \end{align*}
as desired.
So we may prove Theorems \ref{thm2} and \ref{thm3} by replacing $w$ with $w_\ast$.
Now we are in a good light that we can use the property $w_\ast(\cdot,t)\in A_2(\mathbb{R}^n)$
with $C_{A_2}$ which is uniform in almost every $t\in\mathbb{R}$
and independent of $w$.
From this argument, we may assume, for simplicity of notation, that $w$ satisfies the same $A_2$ condition
which enables us to make use of a frequency localization argument on weighted $L^2$ spaces.

By this $A_2$ condition we can indeed use the Littlewood-Paley theorem on weighted $L^2$ spaces (see Theorem 1 in \cite{Ku}
and also Theorem 5 in \cite{Ko}) to get
    \begin{align}\label{p1}
    \nonumber\big\|e^{it(-\Delta)^{\gamma/2}}f\big\|_{L^2(w(x,t))}^2
    &=\int\big\|e^{it(-\Delta)^{\gamma/2}}f\big\|_{L^2(w(\cdot,t))}^2dt\\
    \nonumber&\leq C\int\bigg\| \bigg( \sum_k\big|P_k e^{it(-\Delta)^{\gamma/2}} f \big|^2\bigg)^{1/2} \bigg\|_{L^2(w(\cdot,t))}^2 dt\\
    &=C\sum_k\big\|e^{it(-\Delta)^{\gamma/2}}P_kf\big\|_{L^2(w(x,t))}^2
    \end{align}
for $w\in\mathfrak{L}^{\alpha,p}_{\gamma}$ with $\alpha>\gamma/p$ and $p>1$.
Here the constant $C$ which follows from the Littlewood-Paley theorem is generally depending on the weight $w$
by $C=C_w=C_{A_2}$, but in our case $C_{A_2}$ is independent of $w$.
On the other hand, since $P_kP_jf=0$ if $|j-k|\geq2$, it follows from \eqref{F_local_homo} that
    \begin{align}\label{p2}
    \nonumber\sum_k\big\|e^{it(-\Delta)^{\gamma/2}} P_k f\big\|_{L^2(w(x,t))}^2
    &=\sum_k\big\|e^{it(-\Delta)^{\gamma/2}} P_k \big(\sum_{|j-k|\leq1}P_jf\big)\big\|_{L^2(w(x,t))}^2\\
    &\leq C\|w\|_{\mathfrak{L}^{\alpha,p}_\gamma}\sum_k 2^{k(\alpha -\gamma)} \big\|\sum_{|j-k|\leq1}P_jf\big\|_2^2
    \end{align}
under the same conditions as in Proposition \ref{Prop_F_local}.
By combining \eqref{p1} and \eqref{p2} with $\alpha=2s+\gamma$, we therefore get
 $$\big\|e^{it(-\Delta)^{\gamma/2}} f \big\|_{L^2(w(x,t))}
 \leq C\|w\|_{\mathfrak{L}^{2s+\gamma,p}_\gamma}^{1/2}\|f\|_{\dot{H}^s}$$
under the same conditions as in the theorems.
Hence the proof is now complete.

\section{Proof of Proposition \ref{Prop_F_local}}\label{sec4}

This section is devoted to proving estimate \eqref{Prop_F_local} in Proposition \ref{Prop_F_local}
by combining the $TT^*$ argument (Subsection \ref{subsec4.1}) and the bilinear interpolation (Subsection \ref{subsec4.2})
between its time-localized bilinear form estimates which will be obtained in Subsection \ref{subsec4.3}.

\subsection{$TT^*$ argument}\label{subsec4.1}
From the scaling $(x,t)\rightarrow(\lambda x,\lambda^{\gamma}t)$,
it is enough to show the following case where $k=0$ in \eqref{F_local_homo}:
\begin{equation}\label{freq2}
\big\|e^{it(-\Delta)^{\gamma/2}} P_0 f \big\|_{L^2(w(x,t))}
\leq C\|w\|_{\mathfrak{L}^{\alpha,p}_{\gamma}}^{1/2}\|f\|_{L^2}.
\end{equation}
In fact, once we show this estimate, we get
    \begin{align*}
    \big\|e^{it(-\Delta)^{\gamma/2}}P_k f\big\|_{L^2(w(x,t))}^2
    &\leq C2^{-kn}2^{-\gamma k}\big\|e^{it(-\Delta)^{\gamma/2}}P_0(f(2^{-k}\cdot))\big\|_{L^2(w(2^{-k}x,2^{-\gamma k}t))}^2\\
    &\leq C2^{-kn}2^{-\gamma k}\|w(2^{-k}x,2^{-\gamma k}t)\|_{\mathfrak{L}^{\alpha,p}_{\gamma}}\|f(2^{-k}\cdot)\|_2^2\\
    &\leq C2^{k(\alpha -\gamma)}\|w\|_{\mathfrak{L}^{\alpha,p}_{\gamma}}\|f\|_2^2
    \end{align*}
as desired.
It is then easy to see that the following
$$F\mapsto T^*F:=\int_{-\infty}^\infty e^{-is(-\Delta)^{\gamma/2}}P_0F(\cdot,s)ds$$
is the adjoint operator of
$$f\mapsto Tf:=e^{it(-\Delta)^{\gamma/2}}P_0f.$$
By the usual $TT^*$ argument, \eqref{freq2} is now equivalent to
\begin{equation}\label{eno00}
\bigg\|\int_{-\infty}^{\infty} e^{i(t-s)(-\Delta)^{\gamma/2}}P_0^2F(\cdot,s)ds\bigg\|_{L^2(w(x,t))}
\leq C\|w\|_{\mathfrak{L}^{\alpha,p}_{\gamma}} \|F\|_{L^2(w(x,t)^{-1})};
\end{equation}
however, we shall show a stronger estimate
    \begin{equation}\label{eno}
    \bigg\|\int_{-\infty}^t e^{i(t-s)(-\Delta)^{\gamma/2}}P_0^2F(\cdot,s)ds\bigg\|_{L^2(w(x,t))}
\leq C\|w\|_{\mathfrak{L}^{\alpha,p}_{\gamma}} \|F\|_{L^2(w(x,t)^{-1})}
    \end{equation}
given by replacing $\int_{-\infty}^\infty$ in \eqref{eno00} by $\int_{-\infty}^t$,
which implies the inhomogeneous estimate
\begin{equation}\label{eno0}
    \bigg\|\int_{0}^{t} e^{i(t-s)(-\Delta)^{\gamma/2}}P_0^2F(\cdot,s)ds\bigg\|_{L^2(w(x,t))}
    \leq C\|w\|_{\mathfrak{L}^{\alpha,p}_{\gamma}} \|F\|_{L^2(w(x,t)^{-1})}
    \end{equation}
as well as \eqref{eno00}.
This is the reason why we show \eqref{eno} instead of \eqref{eno00}.
Indeed, to deduce \eqref{eno0} from \eqref{eno},
first decompose the $L_t^2$ norm in the left-hand side of \eqref{eno0}
into two parts, $t\geq0$ and $t<0$. Then the latter can be reduced to the former
by a change of variables $t\mapsto-t$, and so we only need to consider the first part $t\geq0$.
But, since $[0,t)=(-\infty,t)\cap[0,\infty)$, by applying \eqref{eno} with $F$ replaced by $\chi_{[0,\infty)}(s)F$,
the first part follows directly, as desired.
The other estimate \eqref{eno00} follows by a similar argument in which we divide the integral $\int_{-\infty}^\infty$
into two parts, $\int_{-\infty}^t$ and $\int_t^{\infty}$.

\subsection{Bilinear approach}\label{subsec4.2}
To show \eqref{eno}, by duality we may show the following bilinear form estimate
\begin{equation}\label{key_est_1}
    \bigg|\bigg\langle\int_{-\infty}^{t} e^{i(t-s)(-\Delta)^{\gamma/2}} P_0^2 F(\cdot,s)ds,G(x,t)\bigg\rangle_{L_{x,t}^2}\bigg|
    \leq C\|w\|_{\mathfrak{L}^{\alpha,p}_{\gamma}} \|F\|_{L^2(w^{-1})}\|G\|_{L^2(w^{-1})}
\end{equation}
if $p>1$ and
	\begin{equation*}
	\begin{aligned}
	\begin{cases}
	\alpha>1+\frac{n-2+2\gamma}{2p}
	\quad\text{when}\quad \gamma>1,\\
	\alpha>1+\frac{n+1}{2p}
	\quad\text{when}\quad\gamma=1.
	\end{cases}
	\end{aligned}
	\end{equation*}
This form is more flexible so that the estimate follows by bilinear interpolation (Lemma \ref{inter}) between
its time-localized estimates (Proposition \ref{tlocal}).

To do so, we first decompose dyadically \eqref{key_est_1} in time;
for fixed $j\geq1$ and $\gamma\geq1$, define intervals
    $$
    I_j= [2^{j-1}, 2^j)\quad\mbox{and}\quad I_0=[0,1).
    $$
Then we may write
    \begin{align}\label{deco1}
    \nonumber\bigg\langle\int_{-\infty}^{t}e^{i(t-s)(-\Delta)^{\gamma/2}}&P_0^2F(\cdot,s)ds,G(x,t)\bigg\rangle_{L_{x,t}^2} \\
    =\sum_{j\geq0}&\int_{\mathbb{R}}\int_{t- I_j}\Big\langle e^{i(t-s)(-\Delta)^{\gamma/2}}P_0^2 F(\cdot,s),G(x,t)\Big\rangle_{L^2_x}dsdt,
    \end{align}
where $t-I_j=(t-2^j, t-2^{j-1}]$ for $j\geq1$, and $t-I_0=(t-1,t]$.

Next we will obtain the following estimates for these dyadic pieces of \eqref{key_est_1}:

\begin{prop}\label{tlocal}
Let $n\geq1$ and $\gamma\geq1$. Then we have
    \begin{equation}\label{L^2_est}
    \bigg|\int_{\mathbb{R}}\int_{t- I_j}  \Big\langle e^{i(t-s)(-\Delta)^{\gamma/2}} P_0^2 F(\cdot,s),G(x,t)\Big\rangle_{L^2_x}dsdt\bigg|
    \leq C 2^{j} \|F\|_{L^2} \|G\|_{L^2}
    \end{equation}
and
    \begin{equation}\label{L^2_weight_est}
    \begin{aligned}
    \bigg|\int_{\mathbb{R}}\int_{t- I_j}\Big\langle &e^{i(t-s)(-\Delta)^{\gamma/2}} P_0^2 F(\cdot,s),G(x,t)\Big\rangle_{L^2_x}dsdt\bigg|\\
    &\leq CC_\gamma(j) 2^{j(n+\gamma-p\alpha)} \|w\|_{\mathfrak{L}^{\alpha,p}_{\gamma}}^p \|F\|_{L^2(w^{-p})} \|G\|_{L^2(w^{-p})},
    \end{aligned}
    \end{equation}
where $C_{\gamma}(j)=2^{-nj/2}$ for $\gamma>1$, and
$C_{\gamma}(j)=2^{-(n-1)j/2}$ for $\gamma=1$.
\end{prop}

Assuming for the moment this proposition,
we shall now deduce \eqref{key_est_1} from bilinear interpolation between the two estimates \eqref{L^2_est} and \eqref{L^2_weight_est}.
Indeed, let $T$ be a bilinear vector-valued operator defined by
$$T(F,G)=\bigg\{\int_{\mathbb{R}}\int_{t- I_j}\Big\langle e^{i(t-s)(-\Delta)^{\gamma/2}}P_0^2 F(\cdot,s),G(x,t)\Big\rangle_{L^2_x}dsdt\bigg\}_{j\geq0}.$$
If, for example, $\gamma>1$, then \eqref{L^2_est} and \eqref{L^2_weight_est} are equivalent to
$$T:L^2\times L^2\rightarrow \ell_\infty^{s_0}
\quad\text{and}\quad
T:L^2(w^{-p})\times L^2(w^{-p})\rightarrow \ell_\infty^{s_1}$$
with the operator norms $C$ and
$C\|w\|_{\mathfrak{L}^{\alpha,p}_{\gamma}}^p$,
respectively, where
$s_0=-1$ and $s_1=-(n/2+\gamma-p\alpha)$.
Now, by applying Lemma \ref{inter} with $\theta=1/p$, $q=\infty$ and $p_1=p_2=2$, we get
\begin{equation}\label{bibi}
T:(L^2,L^2(w^{-p}))_{1/p,2}\times(L^2,L^2(w^{-p}))_{1/p,2}
\rightarrow(\ell_\infty^{s_0},\ell_\infty^{s_1})_{1/p,\infty}
\end{equation}
with the operator norm
$C\|w\|_{\mathfrak{L}^{\alpha,p}_{\gamma}}$ for $p>1$.
Finally, by applying the real interpolation space identities in Lemma \ref{id} to \eqref{bibi},
one can get
$$T:L^2(w^{-1})\times L^2(w^{-1})\rightarrow\ell_\infty^{s}$$
with
$$s=(1-\frac1p)s_0+\frac1p s_1=-(1+\frac{n-2+2\gamma}{2p}-\alpha)$$
and the operator norm $C\|w\|_{\mathfrak{L}^{\alpha,p}_{\gamma}}$ for $p>1$.
Clearly, this is equivalent to
    $$
    \begin{aligned}
    \bigg|\int_{\mathbb{R}}\int_{t- I_j}\Big\langle e^{i(t-s)(-\Delta)^{\gamma/2}} &P_0^2 F(\cdot,s),G(x,t)\Big\rangle dsdt\bigg|\\
    \leq C&2^{j(1+ \frac{n-2+2\gamma}{2p} - \alpha)} \|w\|_{\mathfrak{L}^{\alpha,p}_{\gamma}} \|F\|_{L^2(w^{-1})}\|G\|_{L^2(w^{-1})}
    \end{aligned}
    $$
for $p>1$.
By summing this over $j\geq0$ and using the decomposition \eqref{deco1}, this implies the bilinear form estimate \eqref{key_est_1}
if\, $p>1$ and $\alpha>1+\frac{n-2+2\gamma}{2p}$ when $\gamma>1$.
Similarly for the case $\gamma=1$.

\subsection{Proof of Proposition \ref{tlocal}}\label{subsec4.3}
Now we prove the estimates \eqref{L^2_est} and \eqref{L^2_weight_est} in Proposition \ref{tlocal}.
Compared with \eqref{L^2_est} which is easy to show, some important features of the proof
of the more delicate estimate \eqref{L^2_weight_est}
can be revealed by further localization in time and space, and then
by analyzing a relevant oscillatory integral under this localization (see \eqref{416} and \eqref{417}).

\subsubsection{Proof of \eqref{L^2_est}}

For fixed $j\geq0$, we first decompose $\mathbb{R}$ into intervals of length $2^j$ to get
    \begin{align}\label{decom_times}
    \nonumber\iint_{t- I_j}&\Big\langle e^{i(t-s)(-\Delta)^{\gamma/2}} P_0^2 F(\cdot,s) , G(x,t) \Big\rangle_{L^2_x} dsdt \\
    &=\sum_{k\in\mathbb{Z}} \int_{2^{j}k}^{2^{j}(k+1)} \int_{t- I_j}
        \Big\langle e^{i(t-s)(-\Delta)^{\gamma/2}} P_0^2 F(\cdot,s) , G(x,t) \Big\rangle_{L^2_x} dsdt.
    \end{align}
Then by H\"older's inequality and Plancherel's theorem in $x$, we see that
    \begin{align}\label{09}
    \nonumber\int_{2^{j}k}^{2^{j}(k+1)} &\int_{t- I_j}
        \Big| \Big\langle e^{i(t-s)(-\Delta)^{\gamma/2}} P_0^2 F(\cdot,s) , G(x,t) \Big\rangle_{L^2_x} \Big|\, dsdt \\
    \nonumber&\leq \int_{2^{j}k}^{2^{j}(k+1)} \int_{t- I_j}
        \big\| e^{i(t-s)(-\Delta)^{\gamma/2}} P_0^2 F(\cdot,s) \big\|_{L^2_x} \|G(\cdot,t)\|_{L^2_x} dsdt \\
    &\leq \int_{2^{j}k}^{2^{j}(k+1)} \int_{2^{j}(k-1)}^{2^{j}(k+\frac12)}
        \|F(\cdot,s)\|_{L^2_y} \|G(\cdot,t)\|_{L^2_x} dsdt .
    \end{align}
On the other hand, we get
    \begin{align}\label{08}
    \nonumber\sum_{k\in\mathbb{Z}} &\int_{2^{j}k}^{2^{j}(k+1)} \int_{2^{j}(k-1)}^{2^{j}(k+\frac12)}
        \|F(\cdot,s)\|_{L^2_y} \|G(\cdot,t)\|_{L^2_x} dsdt \\
    \nonumber&\leq C2^{j}\sum_{k\in\mathbb{Z}}\bigg( \int_{2^{j}(k-1)}^{2^{j}(k+\frac12)}\|F(\cdot,s)\|_{L^2_y}^2 ds\bigg)^{\frac{1}{2}}
        \bigg( \int_{2^{j}k}^{2^{j}(k+1)} \|G(\cdot,t)\|_{L^2_x}^2  dt \bigg)^{\frac{1}{2}} \\
    &\leq C2^{j} \|F\|_{L^2_{x,t}} \|G\|_{L^2_{x,t}}
    \end{align}
by using H\"older's inequality and then the Cauchy-Schwarz inequality in $k$.
Combining \eqref{decom_times}, \eqref{09} and \eqref{08}, we obtain \eqref{L^2_est} as desired.

\subsubsection{Proof of \eqref{L^2_weight_est}}

For fixed $j\geq0$, we first decompose $\mathbb{R}^n$ into cubes of side length $2^j$ to get

$$F(y,s)=\sum_{\lambda\in\mathbb{Z}^n}F_{\lambda}(y,s)\quad\text{and}\quad G(x,t)=\sum_{\lambda\in\mathbb{Z}^n} G_{\lambda}(x,t),$$
where
$$F_{\lambda}(y,s)=\chi_{2^j\lambda +[0,2^j)^n}(y)F(y,s)\quad\text{and}\quad
G_{\lambda}(x,t)=\chi_{2^j\lambda +[0,2^j)^n}(x)G(x,t).$$
Using this decomposition and \eqref{decom_times}, it is enough to show that
    \begin{align}\label{show}
    \nonumber\sum_{k\in\mathbb{Z}} \sum_{\lambda_1, \lambda_2\in\mathbb{Z}^n}
    &\int_{2^{j}k}^{2^{j}(k+1)}\int_{t- I_j}
    \Big|\Big\langle e^{i(t-s)(-\Delta)^{\gamma/2}} P_0^2 F_{\lambda_1}(\cdot,s), G_{\lambda_2}(x,t) \Big\rangle_{L^2_x} \Big| dsdt\\
 &\leq CC_\gamma(j) 2^{j(n+\gamma-p\alpha)} \|w\|_{\mathfrak{L}^{\alpha,p}_{\gamma}}^p \|F\|_{L^2(w^{-p})} \|G\|_{L^2(w^{-p})},
    \end{align}
where $C_{\gamma}(j)=2^{-nj/2}$ for $\gamma>1$, and
$C_{\gamma}(j)=2^{-(n-1)j/2}$ for $\gamma=1$.

To show this, we first write
    \begin{align}\label{kerr}
    \nonumber\Big\langle &e^{i(t-s)(-\Delta)^{\gamma/2}}P_0^2 F_{\lambda_1}(\cdot,s) , G_{\lambda_2}(x,t) \Big\rangle_{L^2_x} \\
    &=\int_{\mathbb{R}^n} \int_{\mathbb{R}^n}
        \bigg(\int_{\mathbb{R}^n} e^{i(x-y)\cdot\xi + i(t-s)|\xi|^\gamma} \phi(|\xi|)^2 d\xi\bigg)
        F_{\lambda_1}(y,s) G_{\lambda_2}(x,t) dydx,
    \end{align}
and then use the following estimates
\begin{equation}\label{416}
    \bigg| \int_{\mathbb{R}^n} e^{i(x-y)\cdot\xi + i(t-s)|\xi|^\gamma} \phi(|\xi|)^2 d\xi \bigg|
    \leq C \big( 2^j |\lambda_1 -\lambda_2| \big)^{-10n}
\end{equation}
when $|\lambda_1-\lambda_2|\geq\gamma2^{\gamma+1}$, and when $|\lambda_1 -\lambda_2|<\gamma2^{\gamma+1}$
\begin{equation}\label{417}
    \bigg| \int_{\mathbb{R}^n} e^{i(x-y)\cdot\xi + i(t-s)|\xi|^\gamma} \phi(|\xi|)^2 d\xi \bigg|
    \leq CC_{\gamma}(j),
\end{equation}
if $x\in2^j\lambda_1+[0,2^j)^n$, $y\in2^j\lambda_2+[0,2^j)^n$ and $s\in t-I_j$.
Indeed, the first estimate follows by integration by parts since
$$|x-y|\geq2^{j-1}|\lambda_1-\lambda_2|\geq|t-s|\gamma2^{\gamma}$$
when $|\lambda_1-\lambda_2|\geq\gamma2^{\gamma+1}$.
When $|\lambda_1-\lambda_2|<\gamma2^{\gamma+1}$,
we use the following well-known stationary phase lemma\footnote{\,It is essentially due to Littman \cite{L}. See also \cite{St}, VIII, Section 5, B.}, Lemma \ref{25}, with $\varphi(\xi)=|\xi|^\gamma$, $\gamma\geq1$, to get \eqref{417}
for $j\geq1$. When applying Lemma \ref{25},
the Hessian matrix $H\varphi$ has $n$ (or, $n-1$) non-zero eigenvalues for each $\xi\in \{\xi\in \mathbb{R}^n: |\xi|\sim1 \}$
when $\gamma>1$ (or, $\gamma=1$),
and note that $|t-s|\geq2^{j-1}$ for $j\geq1$.
Estimate \eqref{417} for $j=0$ is trivial since $C_{\gamma}(j)=1$.

\begin{lem}\label{25}
Let $H\varphi$ be the Hessian matrix given by $(\frac{\partial^2\varphi}{\partial\xi_i\partial\xi_j})$.
Suppose that $\eta$ is a compactly supported smooth function on $\mathbb{R}^n$
and $\varphi$ is a smooth function satisfying rank $H\varphi\geq k$ on the support of $\eta$.
Then, for $(x,t)\in \mathbb{R}^{n+1}$
$$\bigg|\int e^{i(x\cdot\xi+t\varphi(\xi))}\eta(\xi)d\xi\bigg|
\leq C(1+|(x,t)|)^{-\frac{k}{2}}.$$
\end{lem}

Let us now denote
\begin{equation*}
C_{\gamma,\lambda_1,\lambda_2}(j) =
\begin{cases}
\big( 2^j |\lambda_1 -\lambda_2| \big)^{-10n}\quad\text{if}\quad|\lambda_1-\lambda_2|\geq\gamma2^{\gamma+1},\\
C_\gamma(j)\quad\text{if}\quad|\lambda_1 -\lambda_2|<\gamma2^{\gamma+1}.
\end{cases}
\end{equation*}
Using \eqref{kerr}, \eqref{416} and \eqref{417}, we then have
    \begin{align}\label{tty}
   \nonumber\int_{2^{j}k}^{2^{j}(k+1)}& \int_{t- I_j}
        \Big| \Big\langle e^{i(t-s)(-\Delta)^{\gamma/2}} P_0^2 F_{\lambda_1}(\cdot,s) , G_{\lambda_2}(x,t) \Big\rangle_{L^2_x} \Big| ds dt \\
    \nonumber\leq &C C_{\gamma,\lambda_1,\lambda_2}(j)
        \int_{2^{j}k}^{2^{j}(k+1)} \int_{t- I_j}
        \int_{\mathbb{R}^n} \int_{\mathbb{R}^n}
        \big| F_{\lambda_1}(y,s) G_{\lambda_2}(x,t) \big| dydxdsdt \\
    \leq &C C_{\gamma,\lambda_1,\lambda_2}(j)
        \int_{2^{j}k}^{2^{j}(k+1)} \int_{2^{j}(k-1)}^{2^{j}(k+\frac12)} \|F_{\lambda_1}(\cdot,s)\|_{L^1_y} \|G_{\lambda_2}(\cdot,t)\|_{L^1_x} dsdt.
    \end{align}
Notice here that
    $$
    \begin{aligned}
    &\int_{2^{j}(k-1)}^{2^{j}(k+\frac12)}\|F_{\lambda_1}(\cdot,s)\|_{L^1_y} ds \\
    &\quad=\int_{2^{j}(k-1)}^{2^{j}(k+\frac12)} \int_{y\in 2^j \lambda_1 +[0,2^j)^n} |F_{\lambda_1}(y,s)| \cdot |w(y,s)|^{-p/2} \cdot|w(y,s)|^{p/2} dyds \\
    &\quad\leq\|F_{\lambda_1} \chi_{[2^{j}(k-1), 2^{j}(k+\frac12))} \|_{L^2_{y,s}(|w|^{-p})} \bigg( \int_{2^{j}(k-1)}^{2^{j}(k+\frac12)}
    \int_{y\in 2^j \lambda +[0,2^j)^n} |w(y,s)|^{p} dyds \bigg)^{\frac{1}{2}} \\
    &\quad\leq2^{j(\frac{n+\gamma}{p} -\alpha)\frac{p}{2}} \|w\|_{\mathfrak{L}^{\alpha,p}_{\gamma}}^{p/2} \|F_{\lambda_1} \chi_{[2^{j}(k-1), 2^{j}(k+\frac12))} \|_{L^2_{y,s}(|w|^{-p})}.
    \end{aligned}
    $$
Here, for the second inequality we used the definition of the norm $\|w\|_{\mathfrak{L}^{\alpha,p}_{\gamma}}$
together with
$$[2^{j}(k-1),2^{j}(k+\frac12))\subset I_k$$
for some interval $I_k$ of length $2\times2^{j\gamma}$.
Similarly,
    $$
    \int_{2^{j}k}^{2^{j}(k+1)} \|G_{\lambda_2}(\cdot,t)\|_{L^1_x} dt
    \leq 2^{j(\frac{n+\gamma}{p} -\alpha)\frac{p}{2}} \|w\|_{\mathfrak{L}^{\alpha,p}_{\gamma}}^{p/2} \|G_{\lambda_2} \chi_{[2^{j}k, 2^{j}(k+1))} \|_{L^2_{x,t}(|w|^{-p})}.
    $$
Applying these estimates to \eqref{tty} and then using the Cauchy-Schwarz inequality in $k$, one can see that
    \begin{align}\label{thn}
    \nonumber&\sum_{k\in\mathbb{Z}}\sum_{\lambda_1, \lambda_2\in\mathbb{Z}^n}
    \int_{2^{j}k}^{2^{j}(k+1)} \int_{t- I_j}
    \Big|\Big\langle e^{i(t-s)(-\Delta)^{\gamma/2}} P_0^2 F_{\lambda_1}(\cdot,s) , G_{\lambda_2}(x,t) \Big\rangle_{L^2_x} \Big| dsdt \\
    &\leq C 2^{j(\frac{n+\gamma}{p} -\alpha)p} \|w\|_{\mathfrak{L}^{\alpha,p}_{\gamma}}^{p} \sum_{\lambda_1, \lambda_2\in\mathbb{Z}^n}
    C_{\gamma,\lambda_1,\lambda_2}(j)
    \|F_{\lambda_1} \|_{L^2_{y,s}(|w|^{-p})} \|G_{\lambda_2} \|_{L^2_{x,t}(|w|^{-p})}.
    \end{align}

When $|\lambda_1 -\lambda_2|<\gamma2^{\gamma+1}$, the Cauchy-Schwarz inequality in $\lambda_1$ implies
    $$
    \begin{aligned}
    &\sum_{\{\lambda_1, \lambda_2: |\lambda_1 -\lambda_2|<\gamma2^{\gamma+1}\}}
        C_{\gamma,\lambda_1,\lambda_2}(j) \|F_{\lambda_1} \|_{L^2_{y,s}(|w|^{-p})} \|G_{\lambda_2} \|_{L^2_{x,t}(|w|^{-p})} \\
    &\qquad\leq C_\gamma(j) \Big( \sum_{\lambda_1}
        \|F_{\lambda_1} \|_{L^2_{y,s}(|w|^{-p})}^2 \Big)^{\frac{1}{2}}
        \Big( \sum_{\lambda_1} \Big( \sum_{\{\lambda_2:|\lambda_1 -\lambda_2|<\gamma2^{\gamma+1}\}} \|G_{\lambda_2} \|_{L^2_{x,t}(|w|^{-p})} \Big)^2 \Big)^{\frac{1}{2}} \\
    &\qquad\leq C C_\gamma(j) \|F\|_{L^2_{y,s}(|w|^{-p})} \|G\|_{L^2_{x,t}(|w|^{-p})}.
    \end{aligned}
    $$
This and \eqref{thn} imply \eqref{show} as desired.
For the case where $|\lambda_1-\lambda_2|\geq\gamma2^{\gamma+1}$, we first write
     $$
    \begin{aligned}
    \sum_{\{\lambda_1, \lambda_2: |\lambda_1 -\lambda_2|\geq\gamma2^{\gamma+1}\}}&
        C_{\gamma,\lambda_1,\lambda_2}(j) \|F_{\lambda_1} \|_{L^2_{y,s}(|w|^{-p})} \|G_{\lambda_2} \|_{L^2_{x,t}(|w|^{-p})} \\
    =&\sum_{\{\lambda_1, \lambda_2: |\lambda_1 -\lambda_2|\geq\gamma2^{\gamma+1},|\lambda_1|\geq1\}}+
    \sum_{\{\lambda_1, \lambda_2: |\lambda_1 -\lambda_2|\geq\gamma2^{\gamma+1},|\lambda_1|<1\}}.
    \end{aligned}
    $$
By H\"older's inequality in $\lambda_2$ and then Young's inequality, the first term in the right-hand side is then bounded by
  $$
    \begin{aligned}
   &\leq  2^{-10n j} \Big\| \sum_{\{\lambda_1: |\lambda_1|\geq1\}}
   (|\lambda_1 - \lambda_2|)^{-10n}
    \|F_{\lambda_1} \|_{L^2_{y,s}(|w|^{-p})}  \Big\|_{l^2}
        \Big( \sum_{\lambda_2\in\mathbb{Z}^n} \|G_{\lambda_2}\|_{L^2_{x,t}(|w|^{-p})}^2 \Big)^{\frac{1}{2}} \\
   &\leq  2^{-10n j}\Big( \sum_{\{\lambda_1: |\lambda_1|\geq1\}} |\lambda_1|^{-10n}\Big)
        \Big( \sum_{\lambda_1\in\mathbb{Z}^n} \|F_{\lambda_1}\|_{L^2_{x,t}(|w|^{-p})}^2 \Big)^{\frac{1}{2}}
         \Big( \sum_{\lambda_2\in\mathbb{Z}^n} \|G_{\lambda_2}\|_{L^2_{x,t}(|w|^{-p})}^2 \Big)^{\frac{1}{2}} \\
    &\leq C2^{-10n j}\|F\|_{L^2_{y,s}(|w|^{-p})}\|G\|_{L^2_{x,t}(|w|^{-p})},
    \end{aligned}
    $$
while the second term is bounded as follows:
 $$
    \begin{aligned}
   &\sum_{\{\lambda_2: |\lambda_2|\geq\gamma2^{\gamma+1}\}}
        (2^j|\lambda_2|)^{-10n}
        \|F_{0}\|_{L^2_{y,s}(|w|^{-p})} \|G_{\lambda_2}\|_{L^2_{x,t}(|w|^{-p})} \\
   &\qquad\leq2^{-10n j}\|F_{0}\|_{L^2_{y,s}(|w|^{-p})}\Big( \sum_{\{\lambda_2: |\lambda_2|\geq1\}} |\lambda_2|^{-20n}\Big)^{1/2}
         \Big( \sum_{\lambda_2\in\mathbb{Z}^n} \|G_{\lambda_2}\|_{L^2_{x,t}(|w|^{-p})}^2 \Big)^{\frac{1}{2}} \\
    &\qquad\leq C2^{-10n j}\|F\|_{L^2_{y,s}(|w|^{-p})}\|G\|_{L^2_{x,t}(|w|^{-p})}.
    \end{aligned}
    $$
Therefore, we get \eqref{show} in this case as well.

\section{Proofs of the inhomogeneous estimates}\label{sec5}

In this section we prove the inhomogeneous estimates in Propositions \ref{prop}, \ref{prop00} and \ref{prop010}
which are additionally needed to obtain our well-posedness results in the next section.

\subsection{Proof of Proposition \ref{prop}}\label{subsec5.1}

When $\gamma=2$, the inhomogeneous estimate \eqref{higher3} follows from H\"older's inequality
along with the following classical estimate due to Strichartz (\cite{Str2}),
\begin{equation}\label{casd}
\bigg\|\int_{0}^{t}e^{-i(t-s)\Delta} F(\cdot,s)ds\bigg\|_{L^q(\mathbb{R}^{n+1})}\leq C\|F\|_{L^{q'}(\mathbb{R}^{n+1})},
\end{equation}
where $q=\frac{2(n+2)}{n}$.
Indeed,
\begin{align*}
\bigg\|\int_{0}^{t}e^{i(t-s)(-\Delta)^{\gamma/2}} F(\cdot,s)ds\bigg\|_{L^2(w(x,t))}
&\leq \|w^{1/2}\|_{L^{\frac{2q}{q-2}}}\bigg\|\int_{0}^{t}e^{i(t-s)(-\Delta)^{\gamma/2}} F(\cdot,s)ds\bigg\|_{L^q}\\
&\leq C\|w^{1/2}\|_{L^{\frac{2q}{q-2}}}\|F\|_{L^{q'}}\\
&\leq C\|w^{1/2}\|_{L^{\frac{2q}{q-2}}}\|w^{1/2}\|_{L^{\frac{2q'}{2-q'}}}\|w^{-1/2}F\|_{L^{2}}\\
&=C\|w\|_{L^{\frac{n+2}{2}}}\|F\|_{L^{2}(w(x,t)^{-1})}
\end{align*}
as desired. Recall here that $\mathfrak{L}^{2,p}_2=L^p$ with $p=\frac{n+2}{2}$.

On the other hand, the case $\gamma>2$ follows from a similar argument as in the corresponding homogeneous estimate \eqref{higher}.
Indeed, to show \eqref{higher3} in this case, we may first assume that
$w(\cdot,t)\in A_2(\mathbb{R}^n)$ with $C_{A_2}$ which is uniform in almost every $t\in\mathbb{R}$
and independent of $w\in\mathfrak{L}^{\alpha,p}_{\gamma}$, with $\alpha>\gamma/p$ and $p>1$,
as in the homogeneous case (see the first paragraph below Proposition \ref{Prop_F_local}).
Then we shall use the following estimate
	\begin{equation}\label{F_local_inho2}
	\bigg\|\int_{0}^{t}e^{i(t-s)(-\Delta)^{\gamma/2}}P_k F(\cdot,s)ds\bigg\|_{L^2(w(x,t))}
	\leq C2^{k(\alpha -\gamma)}\|w\|_{\mathfrak{L}^{\alpha,p}_{\gamma}} \|F\|_{L^2(w(x,t)^{-1})},
	\end{equation}
where $n\geq1$, $p>1$, $\gamma>1$ and $\alpha>1+\frac{n-2+2\gamma}{2p}$.
Here, $P_k$ is the Littlewood-Paley projection as in Section \ref{sec3}.
Assuming for the moment this estimate, by the Littlewood-Paley theorem on weighted $L^2$ spaces as before,
one can see that
    \begin{align*}
    \bigg\|\int_{0}^{t}e^{i(t-s)(-\Delta)^{\gamma/2}}&F(\cdot,s)ds\bigg\|_{L^2(w(x,t))}^2\\
    \leq \sum_{k}&\bigg\|\int_{0}^{t}e^{i(t-s)(-\Delta)^{\gamma/2}} P_k \big(\sum_{|j-k|\leq1}P_j F(\cdot,s)\big)ds \bigg\|_{L^2(w(x,t))}^2
    \end{align*}
which is in turn bounded by
$$C\|w\|_{\mathfrak{L}^{\alpha,p}_\gamma}^2
\sum_{k} 2^{2k(\alpha -\gamma)} \big\|\sum_{|j-k|\leq1} P_j F\big\|_{L^2(w(x,t)^{-1})}^2$$
if $n\geq1$, $p>1$, $\gamma>1$ and $\alpha>1+\frac{n-2+2\gamma}{2p}$, via \eqref{F_local_inho2}.
Since $w(\cdot,t)^{-1}\in A_2(\mathbb{R}^n)$ if and only if $w(\cdot,t)\in A_2(\mathbb{R}^n)$,
applying the Littlewood-Paley theorem with $\alpha=\gamma$ once more implies
$$\bigg\|\int_{0}^{t}e^{i(t-s)(-\Delta)^{\gamma/2}} F(\cdot,s)ds\bigg\|_{L^2(w(x,t))}
\leq C\|w\|_{\mathfrak{L}^{\gamma,p}_\gamma} \|F\|_{L^2(w(x,t)^{-1})}$$
if $n\geq1$, $\gamma>2$ and $\max\{1,\frac{n+2(\gamma-1)}{2(\gamma-1)}\}<p\leq\frac{n+\gamma}{\gamma}$, as desired.
Now it remains to show \eqref{F_local_inho2} which follows from the case $k=0$ by the scaling,
but we have already obtained the following estimate (see \eqref{eno}) under the same conditions as in \eqref{F_local_inho2}:
\begin{equation}\label{eno8}
    \bigg\|\int_{-\infty}^t e^{i(t-s)(-\Delta)^{\gamma/2}}P_0^2F(\cdot,s)ds\bigg\|_{L^2(w(x,t))}
\leq C\|w\|_{\mathfrak{L}^{\alpha,p}_{\gamma}} \|F\|_{L^2(w(x,t)^{-1})}
    \end{equation}
which is stronger than the desired estimate given by replacing $\int_{-\infty}^t$ in \eqref{eno8} by $\int_{0}^t$.

\subsection{Proof of Proposition \ref{prop00}}
First we rewrite \eqref{F_local_inho2} with $\gamma$ replaced by $\gamma/2$ as
\begin{equation*}
\bigg\|\int_{0}^{t}e^{i(t-s)\sqrt{(-\Delta)^{\gamma/2}}} P_k F(\cdot,s)ds\bigg\|_{L^2(w(x,t))}
\leq C2^{k(\alpha -\frac{\gamma}{2})} \|w\|_{\mathfrak{L}^{\alpha,p}_{\gamma/2}} \|F\|_{L^2(w(x,t)^{-1})},
\end{equation*}
where $n\geq1$, $p>1$, $\gamma>2$ and $\alpha>1+\frac{n-2+\gamma}{2p}$.
Since the Fourier support of $P_k F$ is contained in $\{\xi\in\mathbb{R}^n:|\xi|\sim2^k\}$,
this estimate immediately gives
\begin{equation*}
\bigg\|\int_{0}^{t} \frac{ e^{i(t-s)\sqrt{(-\Delta)^{\gamma/2}}} }{\sqrt{(-\Delta)^{\gamma/2}}} P_k F(\cdot,s)ds\bigg\|_{L^2(w(x,t))}
\leq C2^{k(\alpha -\gamma)} \|w\|_{\mathfrak{L}^{\alpha,p}_{\gamma/2}} \|F\|_{L^2(w(x,t)^{-1})},
\end{equation*}
which in turn implies
\begin{equation*}
\bigg\|\int_{0}^{t} \frac{ e^{i(t-s)\sqrt{(-\Delta)^{\gamma/2}}} }{\sqrt{(-\Delta)^{\gamma/2}}} F(\cdot,s)ds\bigg\|_{L^2(w(x,t))}
\leq C \|w\|_{\mathfrak{L}^{\gamma,p}_{\gamma/2}} \|F\|_{L^2(w(x,t)^{-1})}
\end{equation*}
if $n\geq2$, $2<\gamma<2n$ and $\max\{1,\frac{n-2+\gamma}{2(\gamma-1)}\}<p\leq\frac{2n+\gamma}{2\gamma}$,
by the Littlewood-Paley theorem as in Subsection \ref{subsec5.1}.
Estimate \eqref{121} is now proved except for the case $\gamma=2$.

But, the case $\gamma=2$ follows from the same argument as in Subsection \ref{subsec5.1}
using the classical estimate due to Strichartz (\cite{Str}),
$$\bigg\|\int_{0}^{t} \frac{ e^{i(t-s)\sqrt{-\Delta}} }{\sqrt{-\Delta}} F(\cdot,s)ds\bigg\|_{L^q(\mathbb{R}^{n+1})}\leq C\|F\|_{L^{q'}(\mathbb{R}^{n+1})},$$
where $q=\frac{2(n+1)}{n-1}$ and $n\geq2$, instead of \eqref{casd}.

\subsection{Proof of Proposition \ref{prop010}}

To show the estimate \eqref{fracver} when $\gamma>2$, we first recall from \cite{SW} that
the fractional integral $I_\alpha$ of convolution with $|x|^{-n+\alpha}$, $0<\alpha<n$,
satisfies the inequality
\begin{equation*}
\|I_\alpha f\|_{L^2(w(x))}\leq C\|w\|_{\mathfrak{L}^{\alpha,r}}^{1/2}\|f\|_{L^2},
\end{equation*}
where $\alpha>0$ and $1<r\leq n/\alpha$.
Indeed, this inequality follows directly from (1.10) in \cite{SW} with $w=v^{-1}$ and $p=2$.
Using this inequality and then Plancherel's theorem, we then see that if $\alpha>0$ and $1<r\leq n/\alpha$,
\begin{align*}
\bigg\|\int_{0}^{t}e^{i(t-s)\sqrt{(-\Delta)^{\gamma/2}}}&|\nabla|^{-\alpha}F(\cdot,s)ds\bigg\|_{L_x^2(w(x,t))}\\
\leq C&\|w(\cdot,t)\|_{\mathfrak{L}^{\alpha,r}}^{1/2}\bigg\|\int_{0}^{t}e^{i(t-s)\sqrt{(-\Delta)^{\gamma/2}}}F(\cdot,s)ds\bigg\|_{L_x^2}\\
\leq C&\|w(\cdot,t)\|_{\mathfrak{L}^{\alpha,r}}^{1/2}\bigg\|\int_{\mathbb{R}} e^{-is\sqrt{(-\Delta)^{\gamma/2}}}\chi_{(0,t)}(s)F(\cdot,s)ds\bigg\|_{L_x^2}.
\end{align*}
(Notice here that the integral kernel of the multiplier operator $|\nabla|^{-\alpha}$ is given by $|x|^{-n+\alpha}$, $0<\alpha<n$.)
Combining this estimate and the following dual estimate of \eqref{higher} with $\gamma$ replaced by $\gamma/2$,
\begin{equation}\label{dual3}
\bigg\|\int_{\mathbb{R}} e^{-is\sqrt{(-\Delta)^{\gamma/2}}}F(\cdot,s)ds\bigg\|_{L_x^2}
\leq C\|w\|_{\mathfrak{L}^{2s+\gamma/2,p}_\gamma}^{1/2}\||\nabla|^{s}F\|_{L^2(w(x,t)^{-1})},
\end{equation}
we get
\begin{equation*}
\bigg\|\int_{0}^{t}e^{i(t-s)\sqrt{(-\Delta)^{\gamma/2}}}F(\cdot,s)ds\bigg\|_{L^2(w)}
\leq C\|w\|_{L_t^1\mathfrak{L}^{\alpha,r}}^{1/2}\|w\|_{\mathfrak{L}^{2s+\gamma/2,p}_\gamma}^{1/2}\||\nabla|^{\alpha+s}F\|_{L^2(w^{-1})}
\end{equation*}
for $\alpha>0$, $1<r\leq n/\alpha$, $-\frac{(\gamma-4)n}{4(n+2)} <s< \frac n2$
and $\max\{1,\frac{n+\gamma-2}{4s+\gamma-2}\} <p\leq \frac{2n+\gamma}{4s+\gamma}$.
Applying this estimate with $\alpha+s=\gamma/2$ implies
\begin{equation*}
\bigg\|\int_{0}^{t}\frac{e^{i(t-s)\sqrt{(-\Delta)^{\gamma/2}}}}{\sqrt{(-\Delta)^{\gamma/2}}}F(\cdot,s)ds\bigg\|_{L^2(w(x,t))}
\leq C\|w\|_{L_t^1\mathfrak{L}^{\gamma/2-s,r}}^{1/2}\|w\|_{\mathfrak{L}^{2s+\gamma/2,p}_\gamma}^{1/2}\|F\|_{L^2(w(x,t)^{-1})}
\end{equation*}
for $-\frac{(\gamma-4)n}{4(n+2)}<s<\frac12\min\{\gamma,n\}$, $1<r\leq \frac{2n}{\gamma-2s}$ and $\max\{1,\frac{n+\gamma-2}{4s+\gamma-2}\} <p\leq \frac{2n+\gamma}{4s+\gamma}$ when $2<\gamma<2n+2s$,
as desired.

The case $\gamma=2$ follows from the same argument replacing \eqref{dual3} by the following dual estimate of \eqref{wave},
\begin{equation*}
\bigg\|\int_{\mathbb{R}} e^{-is\sqrt{-\Delta}}F(\cdot,s)ds\bigg\|_{L_x^2}
\leq C\|w\|_{\mathfrak{L}^{2s+1,p}_1}^{1/2}\||\nabla|^{s}F\|_{L^2(w(x,t)^{-1})}.
\end{equation*}

\section{Applications to well-posedness}\label{sec6}

 In this section we present a few applications of our estimates to the global well-posedness of the Schr\"odinger and wave equations
 with small potential perturbations in the weighted $L^2$ setting.

\subsection{The Schr\"odinger equations}
Let us start by considering the following Cauchy problem concerning Schr\"odinger equations of order $\gamma\geq2$
allowing perturbations with time-dependent potentials:
\begin{equation}\label{[Welpo]sch}
\begin{cases}
i\partial_tu-(-\Delta)^{\gamma/2} u +V(x,t)u =F(x,t),\\
u(x,0)=f(x).
\end{cases}
\end{equation}

As we have said before, our interest here is to find a suitable condition on the potentials $V(x,t)$
which guarantees that the problem \eqref{[Welpo]sch} with $L^2$ initial data $f$ is globally well-posed in the weighted $L^2$ space, $L^2(|V|dxdt)$.
Furthermore, it turns out that the solution $u$ belongs to $C_tL_x^2$.
More precisely, we have the following result.

\begin{thm}\label{[Appl]Sch_thm}
Let $n\geq1$ and $\gamma \geq2$.
Assume that $V\in\mathfrak{L}^{\gamma,p}_\gamma$
for $\frac{n+2(\gamma-1)}{2(\gamma-1)} <p\leq \frac{n+\gamma}{\gamma}$ if $\gamma>2$ and for $p=\frac{n+2}{2}$ if $\gamma=2$,
with $\|V\|_{\mathfrak{L}^{\gamma,p}_\gamma}$ small enough.
Then, if $f\in L^2$ and $F\in L^2(|V|^{-1})$, there exists a unique solution of the problem \eqref{[Welpo]sch} in the space $L^2(|V|)$.
	Furthermore, the solution $u$ belongs to $C_tL_x^2$ and satisfies the following inequalities:
	\begin{equation}\label{[Appl]sch:1-1}
	\|u\|_{L^2(|V|)}\leq C\|V\|_{\mathfrak{L}^{\gamma,p}_\gamma}^{1/2}\|f\|_{L^2}+
C\|V\|_{\mathfrak{L}^{\gamma,p}_\gamma}\|F\|_{L^2(|V|^{-1})}
	\end{equation}
	and
	\begin{equation}\label{[Appl]sch:1-2}
	\sup_{t\in\mathbb{R}}\|u\|_{L_x^2}\leq C\|f\|_{L^2}+C\|V\|_{\mathfrak{L}^{\gamma,p}_\gamma}^{1/2}\|F\|_{L^2(|V|^{-1})}
	\end{equation}
\end{thm}

The well-posedness for linear Schr\"odinger equations with potentials has been studied by many authors
(see, for example, \cite{RV,RV2,D'APV,NS,BBRV,S2,KoS2,CKS,KoS3}).
In the context of the weighted $L^2$ setting, it has been handled
in \cite{RV2,BBRV,S2} essentially for small time-independent potentials $V(x)$ of Morrey-Campanato type.
The case of time-dependent potentials was first handled in our previous work \cite{KoS2} where
we obtain Theorem \ref{[Appl]Sch_thm} for higher orders $\gamma>(n+2)/2$
and potentials in a more restrictive class $\mathfrak{L}^{\alpha,\beta,p}$ of Morrey-Campanato type (see \eqref{diff}).
To obtain Theorem \ref{[Appl]Sch_thm} allowing potentials in $\gamma$-order anisotropic Morrey-Campanato classes
$\mathfrak{L}^{\alpha,p}_\gamma$,
which are the most natural ones of Morrey-Campanato type adapted to scaling structure of the Schr\"odinger equations,
we have also attempted in \cite{KoS3} though only for radially symmetric potentials and solutions.
In this regard, the main contribution of the present theorem is to remove assumptions on radial symmetry.

\begin{proof}[Proof of Theorem \ref{[Appl]Sch_thm}]
The proof based on the fixed point argument is rather standard once one has the weighted $L^2$ Strichartz estimates
but we provide a proof for completeness.

The starting point is to write the solution of \eqref{[Welpo]sch} as the integral equation
\begin{equation}\label{[Appl]sol}
u(x,t)= e^{-it(-\Delta)^{\gamma/2}}f(x)-i\int_0^te^{-i(t-s)(-\Delta)^{\gamma/2}}F(\cdot,s)ds+\Phi(u)(x,t),
\end{equation}
where
$$\Phi(u)(x,t)=i\int_0^te^{-i(t-s)(-\Delta)^{\gamma/2}}(Vu)(\cdot,s)ds.$$
Here we note that
$$(I-\Phi)(u)=e^{-it(-\Delta)^{\gamma/2}}f(x)-i\int_0^te^{-i(t-s)(-\Delta)^{\gamma/2}}F(\cdot,s)ds,$$
where $I$ is the identity operator.
Then, since $f\in L^2$ and $F\in L^2(|V|^{-1})$,
by applying the weighted $L^2$ Strichartz estimates in Theorem \ref{thm2} and Proposition \ref{prop} with $w=|V|$,
one can see that
$$(I-\Phi)(u)\in L^2(|V|).$$
Now it is enough to show that
the operator $I-\Phi$ has an inverse in the space $L^2(|V|)$, needed for the fixed point argument.
But, this follows from the fact that the operator norm for $\Phi$ in the space $L^2(|V|)$
is strictly less than $1$.
Namely,
$$\|\Phi(u)\|_{L^2(|V|)}<\frac12\|u\|_{L^2(|V|)}.$$
Indeed, from the inhomogeneous estimate in Proposition \ref{prop} with $w=|V|$, one can see
\begin{align}\label{[Appl]inv}
\|\Phi(u)\|_{L^2(|V|)}
\leq C\|V\|_{\mathfrak{L}^{\gamma,p}_{\gamma}} \|Vu\|_{L^2(|V|^{-1})}
<\frac12\|u\|_{L^2(|V|)}
\end{align}
because of the smallness assumption on the norm $\|V\|_{\mathfrak{L}^{\gamma,p}_{\gamma}}$.

It remains to show \eqref{[Appl]sch:1-1} and \eqref{[Appl]sch:1-2}.
From \eqref{[Appl]sol}, \eqref{[Appl]inv} and Theorem \ref{thm2}, it follows that
\begin{align}\label{[Appl]sch_homo}
\nonumber\|u\|_{L^2(|V|)}
&\leq C\big\|e^{-it(-\Delta)^{\gamma/2}} f \big\|_{L^2(|V|)}
+C\big\|\int_0^te^{-i(t-s)(-\Delta)^{\gamma/2}}F(\cdot,s)ds\big\|_{L^2(|V|)}\\
&\leq C\|V\|_{\mathfrak{L}^{\gamma,p}_\gamma}^{1/2}\|f\|_{L^2}+
C\|V\|_{\mathfrak{L}^{\gamma,p}_\gamma}\|F\|_{L^2(|V|^{-1})}
\end{align}
which is \eqref{[Appl]sch:1-1}.
On the other hand, \eqref{[Appl]sch:1-2} follows from making use of \eqref{[Appl]sch_homo} and the following dual estimate
\begin{equation}\label{[Appl]dual}
\bigg\|\int_{-\infty}^{\infty} e^{is(-\Delta)^{\gamma/2}} F(\cdot,s)ds\bigg\|_{L_x^2}
\leq C\|w\|_{\mathfrak{L}^{\gamma,p}_{\gamma}}^{1/2} \|F\|_{L^2(w(x,t)^{-1})}
\end{equation}
of \eqref{higher} with $s=0$.
Indeed, from \eqref{[Appl]sol}, \eqref{[Appl]dual} with $w=|V|$, and the simple fact that $e^{it(-\Delta)^{\gamma/2}}$ is an isometry in $L^2$,
it follows that
$$\|u\|_{L_x^2}\leq C\|f\|_{L^2}
+C\|V\|_{\mathfrak{L}^{\gamma,p}_{\gamma}}^{1/2} \|F\|_{L^2(|V|^{-1})}
+C\|V\|_{\mathfrak{L}^{\gamma,p}_{\gamma}}^{1/2} \|Vu\|_{L^2(|V|^{-1})}.$$
Since $\|Vu\|_{L^2(|V|^{-1})}=\|u\|_{L^2(|V|)}$
and $\|V\|_{\mathfrak{L}^{\gamma,p}_{\gamma}}$ is small enough,
this and \eqref{[Appl]sch_homo} give
$$\|u\|_{L_x^2}\leq C\|f\|_{L^2}+C\|V\|_{\mathfrak{L}^{\gamma,p}_{\gamma}}^{1/2} \|F\|_{L^2(|V|^{-1})}$$
as desired.
\end{proof}

\subsection{The wave equations}
Next we present the well-posedness results for the following wave equations of order $\gamma\geq2$
as those in the Schr\"odinger setting:
\begin{equation}\label{[Welpo]wav}
\begin{cases}
\partial_t^2u+(-\Delta)^{\gamma/2}u +V(x,t)u=F(x,t),\\
u(x,0)=f(x),\\
\partial_tu(x,0)=g(x).
\end{cases}
\end{equation}

The well-posedness for \eqref{[Welpo]wav} in the space $L_{x,t}^2(|V|)$ was studied only for the wave equation ($\gamma=2$) so far.
In \cite{RV2} it was studied for $V=V_1+V_2$ where $V_1\in L_t^\infty\mathfrak{L}_x^{2,p}$, $(n-1)/2<p\leq n/2$, $V_2\in L_t^rL_x^\infty$, $r>1$,
and $\|V_1\|_{L_t^\infty\mathfrak{L}_x^{2,p}}$ is small enough.
Our results for general orders $\gamma\geq2$ are stated as follows.

\begin{thm}
Let $n\geq2$ and $2\leq\gamma<2n$.
Assume $V\in\mathfrak{L}^{\gamma,p}_{\gamma/2}$ for $\max\{1,\frac{n+\gamma-2}{2(\gamma-1)}\}<p\leq\frac{2n+\gamma}{2\gamma}$
if $\gamma>2$ and for $p=\frac{n+1}{2}$ if $\gamma=2$, with $\|V\|_{\mathfrak{L}^{\gamma,p}_{\gamma/2}}$ small enough.
Then, if $f\in\dot{H}^{\gamma/4}$, $g \in\dot{H}^{-\gamma/4}$ and $F\in L^2(|V|^{-1})$,
	then there exists a unique solution of the problem \eqref{[Welpo]wav} in the space $L^2(|V|)$.
 Furthermore, $u$ and $\partial_tu$ belong to $C_t\dot{H}_x^{\gamma/4}$ and  $C_t\dot{H}^{-\gamma/4}$, respectively,
and satisfies the following inequalities:
	\begin{equation*}
	\|u\|_{L^2(|V|)}\leq C \|V\|_{\mathfrak{L}^{\gamma,p}_{\gamma/2}}^{1/2} \big( \|f\|_{\dot{H}^{\gamma/4}} + \|g\|_{\dot{H}^{-\gamma/4}} \big)
+C\|V\|_{\mathfrak{L}^{\gamma,p}_{\gamma/2}}\|F\|_{L^2(|V|^{-1})}
	\end{equation*}
and
\begin{equation*}
	\sup_{t\in\mathbb{R}}\|u\|_{\dot{H}_x^{\gamma/4}}+\sup_{t\in\mathbb{R}}\|\partial_tu\|_{\dot{H}_x^{-\gamma/4}}
\leq C\big(\|f\|_{\dot{H}^{\gamma/4}} + \|g\|_{\dot{H}^{-\gamma/4}}\big)+C\|V\|_{\mathfrak{L}^{\gamma,p}_{\gamma/2}}^{1/2}\|F\|_{L^2(|V|^{-1})}.
	\end{equation*}
\end{thm}

\begin{proof}
From the estimates \eqref{fcs} and \eqref{higher} with $\gamma$ replaced by $\gamma/2$, one can see that
\begin{equation*}
\|e^{it\sqrt{(-\Delta)^{\gamma/2}}}f\|_{L^2(w(x,t))} \leq C\|w\|_{\mathfrak{L}^{2s+\gamma/2,p}_{\gamma/2}}^{1/2} \|f\|_{\dot{H}^s}
\end{equation*}
if $-\frac{(\gamma-4)n}{4(n+2)} <s< \frac n2$ and $\max\{1,\frac{n+\gamma-2}{4s+\gamma-2}\} <p\leq \frac{2n+\gamma}{4s+\gamma}$
when $\gamma>2$, and if $1/2\leq s<n/2$ and $p=\frac{n+1}{2s+1}$ when $\gamma=2$.
Using this homogeneous estimate with $s=\gamma/4$ and the inhomogeneous estimate \eqref{121},
the proof follows from an obvious modification of those for the Schr\"odinger case.
So we omit the details.
\end{proof}

\begin{rem}
In \cite{KoS} we studied the case $\gamma=2$ in \eqref{[Welpo]wav} for $n=2,3$ when
$V\in L_t^1\mathfrak{L}_x^{1-s,r}\cap\mathfrak{L}_{x,t}^{2s+1,p}$ for $3/4\leq s<1$, $1<r\leq 2/(1-s)$, $1<p\leq(n+1)/(2s+1)$,
with corresponding small norms on $V$. (Here, when $n=3$, $\mathfrak{L}_x^{1-s,r}$ is replaced by $L_x^{\infty}$.)
But, if we substitute in the proof Proposition \ref{prop00} by Proposition \ref{prop010},
we can use our approach, as we did in \cite{KoS}, to improve significantly this previous result.
We do not want to get involved in these calculations in the present paper.
\end{rem}


\section{Appendix}\label{sec7}
\subsection{Proof of Proposition \ref{prop0}}
Let $M$ be a positive constant as large as we need, and $\varphi\in C_0^{\infty}(\mathbb{R})$
be supported in $[-1,1]$ and $0\leq\varphi\leq1$.
We consider a function $f$ such that
\begin{equation*}
\widehat{f}(\xi) = \varphi(\xi_1 -M) \prod_{k=2}^{n} \varphi(\xi_k)
\end{equation*}
which is supported in a cube centered at $(M,0,...,0)$ with side length $\sim1$.
Next, we define the following set
$$
B = \{ (x,t)\in\mathbb{R}^{n+1}: |t|\leq\frac1{4n},\,\, |x_1-2Mt|\leq\frac1{4n},\,\, |x_k|\leq\frac1{4n}\text{ for } k=2,...,n\}.
$$
(See Figure 3.)

\begin{figure}[t!]
\includegraphics[width=12.0cm]{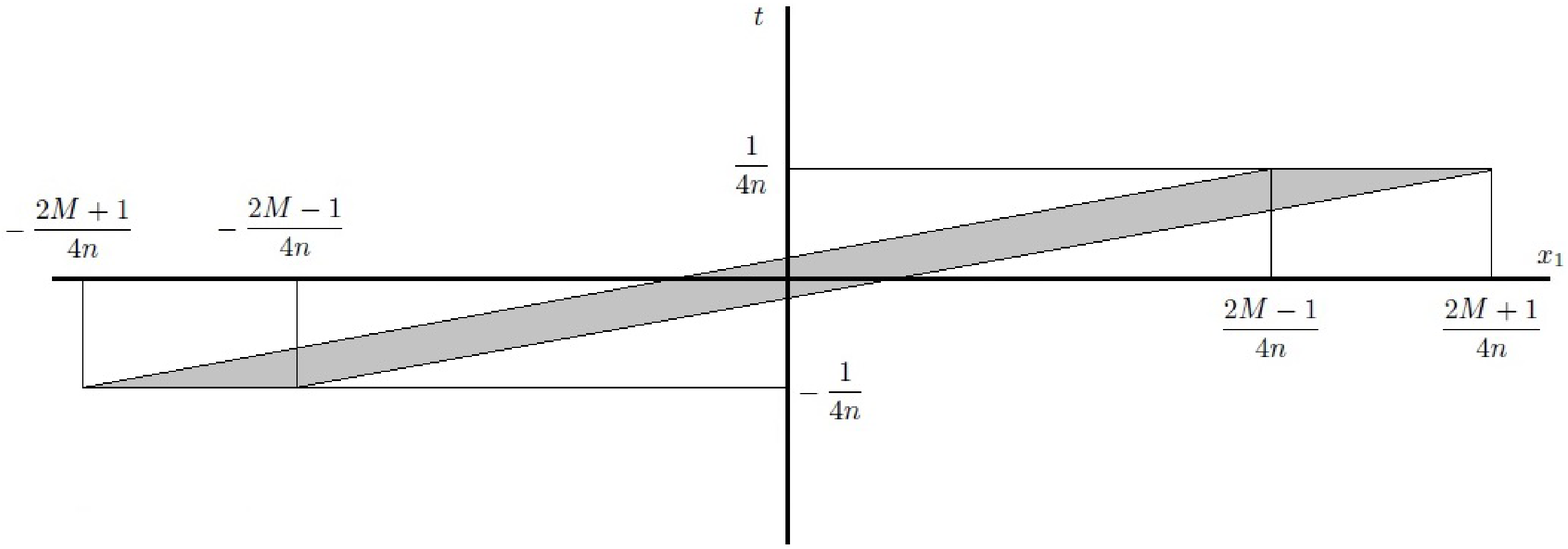}
\caption{The cross section of the set $B$ in the $x_1t$-plane}
\end{figure}

Using a change of variable, $\xi_1\rightarrow\xi_1+M$, we may now write for $(x,t) \in B$
$$
\begin{aligned}
\big||\nabla|^{-s} e^{it\Delta} f(x) \big|
&\sim \bigg|\int_{\mathbb{R}^n}|\xi|^{-s} e^{i x\cdot\xi-it|\xi|^2} \varphi(\xi_1 -M)\prod_{k=2}^{n}\varphi(\xi_k)d\xi \bigg|\\
&\sim \bigg|\int_{\mathbb{R}^n}|\xi|^{-s} e^{i(x_1-2Mt)\xi_1-it\xi_1^2}e^{\sum_{k=2}^n(ix_k\xi_k-it\xi_k^2)}\prod_{k=1}^{n}\varphi(\xi_k)d\xi\bigg|\\
&\gtrsim M^{-s}
\end{aligned}
$$
because the phase is bounded as follows:
$$\big|(x_1-2Mt)\xi_1-t\xi_1^2+\sum_{k=2}^n(x_k\xi_k-t\xi_k^2)\big|\leq1/2.$$
Setting $w(x,t)=\chi_B(x,t)$, we then get
$$
\big\||\nabla|^{-s} e^{it\Delta} f\big\|_{L^2(w(x,t))}
\gtrsim M^{-s},
$$
and
$$\|w\|_{\mathfrak{L}^{\alpha,p}_2}\sim\max\{M^{-\alpha/2},M^{-1/p},M^{\alpha -(n+2)/p}\}$$
which follows by taking $r\in\{M^{-1/2},1,M\}$ in the definition of $\mathfrak{L}^{\alpha,p}_2$.
On the other hand, it is clear that $\|f\|_2 \sim 1$.
Consequently, \eqref{depen2} would imply
$$M^{-2s}\lesssim \max\{M^{-s+1},M^{-1/p},M^{2s+2-(n+2)/p}\}$$
which is false if
$-2s>-s+1$, $-2s>-1/p$ and $-2s>2s+2-(n+2)/p$.
This completes the proof.

\subsection{Further remarks}\label{subsec7.2}

Let us finally mention further estimates for the linearized KdV type equations:
\begin{equation*}
\begin{cases}
\partial_tu+\partial_x^{2k+1}u=F(x,t),\\
u(x,0)=f(x),
\end{cases}
\end{equation*}
whose the solution is given by
\begin{equation*}
u(x,t)=e^{-t\partial_x^{2k+1}}f(x)+\int_0^te^{-(t-s)\partial_x^{2k+1}}F(\cdot,s)ds
\end{equation*}
where $(x,t)\in\mathbb{R}^{1+1}$ and $k\geq1$ is an integer.
The proofs for Theorem \ref{thm2} and Proposition \ref{prop} when $n=1$ and $\gamma=2k+1$ would be clearly worked
for the following proposition as well:

\begin{thm}\label{thm33}
Let $k\geq1$ be an integer. Then we have
\begin{equation}\label{kdv0}
\big\|e^{-t\partial_x^{2k+1}}f \big\|_{L^2(w(x,t))}\leq  C\|w\|_{\mathfrak{L}_{2k+1}^{2(s+k)+1,p}}^{1/2}\|f\|_{\dot{H}^s}
\end{equation}
if $-\frac{2k-1}{6}<s<\frac12$ and $\max\{1,\frac{4k+1}{4(s+k)}\} <p\leq \frac{2k+2}{2(s+k)+1}$,
and
\begin{equation*}
\bigg\|\int_{0}^{t}e^{-(t-s)\partial_x^{2k+1}}F(\cdot,s)ds\bigg\|_{L^2(w(x,t))}
\leq C\|w\|_{\mathfrak{L}_{2k+1}^{2k+1,p}} \|F\|_{L^2(w(x,t)^{-1})}
\end{equation*}
if $\frac{1+4k}{4k}<p\leq\frac{2k+2}{2k+1}$.
\end{thm}

\begin{rem}
The condition on $s$ and $p$ in this theorem is, of course,
the same as in Theorem \ref{thm2} and Proposition \ref{prop}, with $\gamma=2k+1$ and $n=1$,
and we have a smoothing effect in \eqref{kdv0} with a gain of regularity of order $s<\frac{2k-1}{6}$.
\end{rem}

As an application, Theorem \ref{thm33} will then give a similar result on the well-posedness for the following Cauchy problem
\begin{equation*}
\begin{cases}
\partial_tu+\partial_x^{2k+1}u+V(x,t)u=F(x,t),\\
u(x,0)=f(x).
\end{cases}
\end{equation*}


\end{document}